\documentclass[11pt,twoside]{amsart}

\numberwithin{equation}{section}
\usepackage{latexsym,amssymb,amsmath,amsthm,amsfonts,amscd}
\usepackage{enumerate}
\usepackage[table]{xcolor}
\usepackage{graphicx}
\usepackage{tikz}
\usetikzlibrary{matrix,arrows}
\usepackage{multicol}
\usepackage{multirow}

\definecolor{VerdeOlivo}{rgb}{0.3,0.5,0.1}
\definecolor{Magenta}{rgb}{.65,0.15,.2}
\definecolor{Gris}{gray}{0.3}

\newtheorem{Theorem}{Theorem}[section] 
\newtheorem{Definition}[Theorem]{Definition}
\newtheorem{Proposition}[Theorem]{Proposition}  
\newtheorem{Lemma}[Theorem]{Lemma} 
\newtheorem{Corollary}[Theorem]{Corollary}
\newtheorem{Remark}[Theorem]{Remark}
\newtheorem{Example}[Theorem]{Example}
\newtheorem{Claim}[Theorem]{Claim}    
\newtheorem{Conjecture}[Theorem]{Conjecture}

\theoremstyle{definition}

\begin{document} 


\title[On the critical ideals of graphs]{On the critical ideals of graphs}


\author{Hugo Corrales}
\author{Carlos E. Valencia}
\address{
Departamento de Matem\'aticas\\
Centro de Investigaci\'on y de Estudios Avanzados del IPN\\
Apartado Postal 14--740\\
07000 Mexico City, D.F. 
} 
\email[Hugo~Corrales]{hhcorrales@gmail.com}
\email[Carlos E.~Valencia\footnote{Corresponding author}]{cvalencia@math.cinvestav.edu.mx, cvalencia75@gmail.com}
\thanks{Both authors was partially supported by CONACyT grant 166059. 
The first author was partially supported by CONACyT and the second author was partially supported by SNI}


\begin{abstract} 
We introduce some determinantal ideals of the generalized Laplacian matrix associated to a digraph $G$, that we call critical ideals of $G$.
Critical ideals generalize the critical group and the characteristic polynomials of the adjacency and Laplacian matrices of a digraph.
The main results of this article are the determination of some minimal generator sets 
and the reduced Gr\"obner basis for the critical ideals of the complete graphs, the cycles and the paths.
Also, we establish a bound between the number of trivial critical ideals and the stability and clique numbers of a graph.
\end{abstract}

\maketitle


\section{Introduction} 
Let $G=(V,E)$ be a digraph with $n$ vertices, $m_{(u,v)}$ the number of directed arcs from $u$ to $v$,
and $A(G)$ be the \emph{adjacency matrix} of $G$ given by $A(G)_{u,v}=m_{(u,v)}$.
The \emph{Laplacian matrix} of $G$ is given by 
\[
L(G)=D^+(G)-A(G),
\] 
where $D^+(G)$ is the diagonal matrix with the out-degrees of the vertices of $G$ in the diagonal entries.
Note that the Laplacian matrix is insensitive to loops in the digraph.
If $X_G=\{x_u\, |\, u\in V(G)\}$ is the set of variables indexed by the vertices of the digraph $G$,
then the {\it generalized Laplacian matrix} of $G$, denoted by $L(G,X_G)$, 
is given by
\[
L(G,X_G)_{u,v}=\begin{cases}
x_u& \text{ if } u=v,\\
-m_{(u,v)} & \text{ otherwise}.
\end{cases}
\]

Godsil and Royle in~\cite[13.9]{godsil} give a similar way to define a generalized Laplacian matrix.
Furthermore, if $\mathcal{P}$ is a commutative ring with identity, $\mathcal{P}[X_G]$ is the polynomial ring 
over $\mathcal{P}$ in the variables $X_G$ and $1\leq i\leq n$, then the $i$-{\it critical ideal} of $G$ is the determinantal ideal given by
\[
I_i(G,X_G)=\langle {\rm minors}_i(L(G,X_G))\rangle\subseteq \mathcal{P}[X_G],
\]
where ${\rm minors}_i(L(G,X_G))$ is the set of determinants of all the $i$-square submatrices of $L(G,X_G)$.
Note that we can define (without any  technical issue) the critical ideals of an $n\times n$ matrix $M$ with entries in $\mathcal{P}$
as $I_i(M,X)=\langle {\rm minors}_i(L(M,X))\rangle\subseteq \mathcal{P}[X]$ for all $i=1,\ldots,n$, where 
\[
L(M,X)_{u,v}=\begin{cases}
x_u& \text{ if } u=v,\\
-M_{u,v} & \text{ otherwise}.
\end{cases}
\]

In \cite{ideal} was defined a binomial ideal, called the {\it toppling ideal}, associated to the Laplacian matrix of a graph, 
this ideal encodes the topplings of the abelian sandpile model.

When $G$ is a simple connected graph, its {\it critical group}, denoted by $K(G)$, 
is defined as the torsion subgroup of the cokernel of $L(G)$; for instance see~\cite{lorenzini91}.
That is, 
\[
K(G)\oplus\mathbb{Z}=\mathbb{Z}^{V}/{\rm Im}\, L(G)^t.
\]
Wagner in~\cite{directed} defined the critical group for a digraph.
Also, in~\cite{simplicial} it was defined the critical group for simplicial complexes.

As the reader can imagine, the critical ideals and the critical group of a graph are closely related, 
for instance the reader can see Propositions~\ref{correspondence0} and~\ref{correspondence}. 
Moreover, critical ideals of a graph are very useful to get a better understanding of its critical group.
In Section~\ref{critical}, we will show that the critical ideals are better behaved than the critical group.
For instance, if $\gamma_{\mathcal{P}}(G)$ is the number of critical ideals over the base ring $\mathcal{P}$ that are trivial, then Theorem~\ref{bounds} asserts that 
\[
\gamma_{\mathcal{P}}(G)\leq 2(|V(G)|-\alpha(G))\text{ and }\gamma_{\mathcal{P}}(G)\leq 2(|V(G)|-\omega(G))+1,
\] 
where $\alpha(G)$ is the stability number and $\omega(G)$ is the clique number of the graph.
That is; the invariant $\gamma_{\mathcal{P}}$ is closely related to the combinatorics of the graph.
Also, if $H$ is an induced subdigraph of $G$, then $\gamma_\mathcal{P}(H)\leq \gamma_\mathcal{P}(G)$
in contrast with the behavior of the number of invariant factors equal to one of the critical group of induced subdigraphs, see Section~\ref{gamma}.  

\medskip

The main goal of this article is the study of the critical ideals of graphs.
The main results of this article are contained in Sections~\ref{critical} and~\ref{combinatorial}.
Section~\ref{critical} contains the basic properties of the critical ideals and
we define the invariant $\gamma_{\mathcal{P}}$ as the number of the critical ideals that are trivial. 
Also, in this section we get the reduced Gr\"obner basis of the complete graph.
Finally, we explore the relationship of the critical ideals of a graph with the characteristic polynomial of its adjacency and Laplacian matrices.
The reduced Gr\"obner basis of the critical ideals  of the cycles and
some combinatorial expression for the minors of the generalized Laplacian matrix of a digraph
are presented in Section~\ref{combinatorial}. 


\section{Preliminaries} 
A {\it graph} $G$ is a pair $(V, E)$, where $V$ is a finite set and $E$ is a collection of unordered pairs of elements of $V$.
The elements of $V$ are called {\it vertices} and the elements of $E$ are called {\it edges}.
For simplicity, sometimes an edge $e=\{x,y\}$ will be denoted by $xy$. 
The sets of two or more edges with the same ends are called \textit{multiple edges}.
A \textit{loop} is an edge incident to a unique vertex.
A \textit{multigraph} is a graph with multiple edges and no loops. 

A \textit{digraph} $D$ is a pair $(V,E)$, where $V$ is a finite set and $E$ is a set of ordered pairs of elements of $V$.
The elements of $V$ and $E$ are called \textit{vertices} and \textit{arcs}, respectively. 
Given an arc $e=(x,y)$, we say that $x$ is the initial vertex of $e$ and $y$ is the terminal vertex of $e$.
Sometimes for simplicity an arc $(u,v)$ will be denoted by $\overset{\longrightarrow}{uv}$.
In this article, any digraph may contain multiple arcs and loops, unless otherwise is specified.

The number of arcs with initial vertex $x$ and terminal vertex $y$ will be denoted by $m_{(x,y)}$.
The \textit{out-degree} of a vertex $x$ of a digraph $D$, denoted by $d_{D}^+(x)$, is the number of arcs in $D$ with initial vertex $x$.
Given a subset $U$ of the vertices of a digraph $G$, the induced subdigraph by $U$, denoted by $G[U]$, is the subdigraph of 
$G$ that has $U$ as vertex set and $E=\{(u,v)\, | \, u,v\in U \text{ and }(u,v)\in E(G) \}$ as arc set.
We say that a subdigraph $H$ of $G$ is induced if $H=G[U]$ for some $U\subset V(G)$.
In a similar way, we can define an induced subgraph.

Note that there exists a natural inclusion of the graphs into the digraphs. 
Actually, any graph $G$ can be identified with a digraph if we consider each edge $\{u,v\}$ of $G$
as the pair of arcs $(u,v)$ and $(v,u)$.
The reader can consult~\cite{diestel} and~\cite{digraphs} for any unexplained concept of graph and digraph theory, respectively.

\medskip

Now we will introduce some notation for the minors of a matrix.
Let $M\in M_{n}(\mathcal{P})$ be an $n\times n$ matrix with entries on $\mathcal{P}$, $I=\{i_1,\ldots,i_r\}\subseteq [n]$, and $J=\{j_1,\ldots,j_s\}\subseteq [n]$.
The submatrix of $M$ formed by rows $i_1,\ldots,i_r$ and columns $j_1,\ldots,j_s$ is denoted by $M[I;J]$.
On the other hand, the submatrix obtained from $M$ by deleting rows $i_1,\ldots,i_r$ 
and columns $j_1,\ldots,j_s$ will be denoted by $M(I;J)$. 
That is, $M(I;J)=M[I^c;J^c]$.
If $|I|=|J|=r$, then $M[I;J]$ is called an {\it $r$-square submatrix} or a {\it square submatrix} of size $r$ of $M$.
An {\it $r$-minor} is the determinant of an $r$-square submatrix.
The set of $i$-minors of a matrix $M$ will be denoted by ${\rm minors}_i(M)$.
We say that $M,N \in M_{n}(\mathcal{P})$ 
are {\it equivalent}, denoted by $N\sim M$, if there exist invertible matrices $P,Q\in GL_n(\mathcal{P})$ such that $N=PMQ$.
It is not difficult to see that if $N\sim M$, then $K(M)=\mathcal{P}^n/M^t\mathcal{P}^n\cong\mathcal{P}^n/N^t\mathcal{P}^n=K(N)$.

To finish this section, we will recall some useful results on Gr\"obner basis.

\subsection{Gr\"obner Basis}
Usually the theory of Gr\"obner basis deals with ideals in a polynomial ring over a field.
However, in this paper we deal with ideals in a polynomial ring over the integers.
There exists a theory of Gr\"obner basis over almost any kind of rings.
 
We recall some basic concepts on Gr\"obner basis, for more details see~\cite{Lou}. 
First, an {\it order term} in the polynomial ring $R=\mathcal{P}[x_1,\ldots,x_n]$ is a total order $\prec$ in the set of monomials of $R$ such that
\begin{description}
\item[(i)] $1\prec x^{\alpha}$  for all ${\bf 0} \neq {\bf \alpha} \in\mathbb{N}^n$, and  
\item[(ii)] if $x^{\alpha}\prec x^{\beta}$, then 
$x^{\alpha+\gamma}\prec x^{\beta+\gamma}$
for all $\gamma \in\mathbb{N}^n$,
\end{description}
where ${\bf x}^{\alpha}=x_1^{\alpha_1} \cdots x_n^{\alpha_n}$.

Now, given an order term $\prec$ and $p\in \mathcal{P}[X]$, let ${\rm lt}(p)$, ${\rm lp}(p)$, and ${\rm lc}(p)$ 
be the {\it leading term}, the {\it leading power}, and the {\it leading coefficient} of $p$, respectively.
Given a subset $S$ of $\mathcal{P}[X]$ its leading term ideal of $S$ is the ideal
\[
{\rm Lt}(S)=\langle {\rm lt}(s) \,|\, s\in S \rangle.
\]

A finite set of nonzero polynomials $B=\{b_1,\ldots, b_s\}$ of an ideal $I$ is called a {\it Gr\"obner basis} of $I$ 
with respect to an order term $\prec$ if ${\rm Lt}(B)={\rm Lt}(I)$.
Moreover, it is called {\it reduced} if ${\rm lc}(b_i)=1$ for all $1\leq i\leq s$ and no nonzero term in $b_i$
is divisible by any ${\rm lp}(b_j)$ for all $1\leq i\neq j\leq s$.

A good characterization of Gr\"obner basis is given in terms of $S$-polynomials.
\begin{Definition}
Let $f$, $f'$ be polynomials in $\mathcal{P}[X]$ and $B$ be a set of polynomials in $\mathcal{P}[X]$. 
We say that $f$ {\it reduces strongly} to $f'$ modulo $B$ if
\begin{itemize}
\item ${\rm lt}(f')\prec {\rm lt}(f)$, and 
\item there exist $b\in B$ and $h\in \mathcal{P}[X]$ such that $f'=f-hb$. 
\end{itemize}
Moreover, if $f^*\in \mathcal{P}[X]$ can be obtained from $f$ in a finite number of reductions, we write $f\rightarrow_B f^*$.
\end{Definition}

That is, if $f=\sum_{j=1}^t p_{i_j}b_{i_j}+f^*$ with $p_{i_j} \in \mathcal{P}[X]$ and 
${\rm lt}(p_{i_j}b_{i_j})\neq {\rm lt}(p_{i_k}b_{i_k})$ for all $j\neq k$, then $f\rightarrow_B f^*$.

Now, given $f$ and $g$ polynomials in $\mathcal{P}[X]$, their {\it $S$-polynomial}, denoted by $S(f,g)$, is given by
\[
S(f,g)=\frac{c}{c_f}\frac{X}{X_f}\,f-\frac{c}{c_g}\frac{X}{X_g}\,g,
\]
where $X_f=lt(f)$, $c_f=lc(f)$, $X_g=lt(g)$, $c_g=lc(g)$, $X=\textrm{lcm}(X_f,X_g)$, and $c=\textrm{lcm}(c_f,c_g)$.

\medskip

The next Lemma gives us a useful criterion for checking whether a set of generators of an ideal is a  Gr\"obner basis.
\begin{Lemma}\label{GBT}
Let $I$ be an ideal of polynomials over a PID and $B$ be a generating set of $I$. 
Then $B$ is a Gr\"obner basis for $I$ if and only if $S(f,g)\rightarrow_B 0$  for all $f\neq g\in B$.
\end{Lemma}


\section{The critical ideals of graphs}\label{critical} 
In this section, we introduce the main concept of this article: the critical ideals of a digraph $G$.
We will begin this section by defining the critical ideals of a digraph, presenting some examples 
and discussing some of their basic properties.
In general terms, the $i$-th critical ideal is the determinantal ideal of $i$-minors of the generalized Laplacian matrix associated to $G$.
The critical ideals of $G$ generalize the critical group of $G$ (see Proposition~\ref{correspondence0})
and the characteristic polynomials of the adjacency matrix and the Laplacian matrix associated to $G$ (see Section~\ref{characteristic}).
Moreover, with some additional requirements over $G$ we can get a stronger 
correspondence between the critical ideals of $G$ and the critical group of $G$, see Proposition~\ref{correspondence}.
 
Afterwards, in Section~\ref{gamma} we introduce the number of critical ideals that are trivial as an invariant of the digraph.
In the case of graphs, we will establish a bound between this invariant and the stability and clique numbers of the graph. 
Also, in Section~\ref{complete} we present a minimal set of generators and a reduced Gr\"obner basis for the 
critical ideals of the complete graphs. 
As a byproduct we will get expressions for the critical groups for a complete graph minus a star.
Finally, we will explore the relation between the critical ideals of a graph and the characteristic polynomial of its adjacency and Laplacian matrix.

\begin{Definition}
Given a digraph $G$ with $n$ vertices and $1\leq i \leq n$, let
\[
I_i(G,X)=\langle {\rm minors}_i(L(G,X))\rangle\subseteq \mathcal{P}[X_G]
\]
be the $i$-th critical ideal of $G$.
\end{Definition}
Note that in general the critical ideals depend on the base ring $\mathcal{P}$,
in this article we are mainly interested when $\mathcal{P}=\mathbb{Z}$.
By convention, $I_i(G,X)=\langle 1\rangle$ if $i\leq 0$ and $I_i(G,X)=\langle 0\rangle$ if $i> n$.
Clearly $I_n(G,X)$ is a principal ideal generated by the determinant of  the generalized Laplacian matrix.

Now, we present an example that illustrates the concept of critical ideal.

\begin{Example}\label{ejemplo1}
Let $H$ be the complete graph with six vertices minus the perfect matching formed by the edges 
$M_3=\{v_1v_4, v_2v_5, v_3v_6\}$ (see Figure~\ref{fig1}(a)) and $\mathcal{P}=\mathbb{Z}$.
Then, 
\[
L(H,X)=\left[\begin{array}{cccccc}
x_1&-1&-1&0&-1&-1\\
-1&x_2&-1&-1&0&-1\\
-1&-1&x_3&-1&-1&0\\
0&-1&-1&x_4&-1&-1\\
-1&0&-1&-1&x_5&-1\\
-1&-1&0&-1&-1&x_6
\end{array}\right]
\]
Using any algebraic system, for instance {\it Macaulay 2}, it is not difficult to see that $I_i(H,X)=\langle 1 \rangle $ for $i=1,2$ and
\[
I_i(H,X)=\begin{cases}
\big\langle  2,x_1,x_2,x_3,x_4,x_5,x_6\big\rangle & \text{ if } i=3,\\
\big\langle \{x_rx_s \, | \, v_rv_s\in E(H)\}\cup\{2x_r+2x_s+x_rx_s\, | \, v_rv_s \not\in E(H)\}\,\big\rangle & \text{ if } i=4,\\
\big\langle  \{x_kx_l(x_r+x_s+x_rx_s)\, | \, (r,s,k,l)\in S(H)\}\cup\{p_{(r,s,k,l)}\, | \, v_rv_s,v_kv_l\not\in E(H) \} \big\rangle & \text{ if } i=5,\\
\big\langle x_1x_2x_3x_4x_5x_6 -\sum_{(r,s,k,l)\in S(H)} x_rx_sx_kx_l - 2\sum_{(r,s,k)\in T(H)} x_rx_sx_k \big\rangle & \text{ if } i=6,
\end{cases}
\]
where 
\[
S(H)=\{(r,s,k,l)\, | \, v_rv_s\not\in E(H),\ v_kv_l\in E(H),\text{ and } \{i,j\}\cap\{k,l\}=\emptyset\},
\]
$T(H)$ are the triangles of $H$, and $p_{(r,s,k,l)}=(x_r+x_s)(x_k+x_l+x_kx_l)+(x_k+x_l)(x_r+x_s+x_rx_s)$.
Note that the expressions of the critical ideals of $H$ depend heavily on their combinatorics.
\end{Example}

Now, let us turn our attention to one of the most basic properties of the critical ideals.

\begin{Proposition}\label{BasicProp}
If $G$ is a digraph with $n$ vertices, then 
\[
\langle 0\rangle \subsetneq I_n(G,X)\subseteq \cdots \subseteq I_2(G,X)\subseteq I_1(G,X) \subseteq \langle 1\rangle.
\]
Moreover, if $H$ is an induced subdigraph of $G$ with $m$ vertices, then $I_k(H,X)\subseteq I_k(G,X) \text{ for all } 1 \leq k \leq m$.
\end{Proposition}
\begin{proof}
Let $M$ be a $(k+1)\times (k+1)$ matrix with entries on $\mathcal{P}[X_G]$.
Since $\mathrm{det}(M)=\sum_{i=1}^{k+1} M_{i,1}\,\mathrm{det}(M({i};{1}))$, $I_{k+1}(G,X)\subseteq I_k(G,X)$ for all $1 \leq k \leq n-1$.  
On the other hand, since any submatrix of $L(H,X)$ is also a submatrix of $L(G,X)$, 
${\rm minors}_k(L(H,X))\subseteq {\rm minors}_k(L(G,X))$ for all $1 \leq k\leq m$ and therefore $I_k(H,X)\subseteq I_k(G,X)$ for all $1 \leq k\leq m$.
\end{proof}

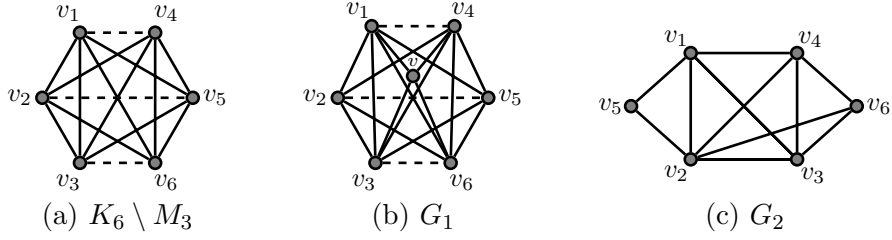
\begin{figure}[h]
\begin{center}
\begin{tabular}{c@{\extracolsep{1cm}}c@{\extracolsep{1cm}}c}
	\begin{tikzpicture}[line width=1pt, scale=1]
		\tikzstyle{every node}=[inner sep=0pt, minimum width=4.5pt] 
		\draw (0:1) node (v5) [draw, circle, fill=gray] {};
		\draw (60:1) node (v4) [draw, circle, fill=gray] {};
		\draw (120:1) node (v1) [draw, circle, fill=gray] {};
		\draw (180:1) node (v2) [draw, circle, fill=gray] {};
		\draw (240:1) node (v3) [draw, circle, fill=gray] {};
		\draw (300:1) node (v6) [draw, circle, fill=gray] {};
		\draw (0:1.3) node () {\small $v_5$};
		\draw (60:1.3) node () {\small $v_4$};
		\draw (120:1.3) node () {\small $v_1$};
		\draw (180:1.3) node () {\small $v_2$};
		\draw (240:1.3) node () {\small $v_3$};
		\draw (300:1.3) node () {\small $v_6$};
		\draw (v1) -- (v2) -- (v3) -- (v5) -- (v1) -- (v3) -- (v4) -- (v5) -- (v6) -- (v4) -- (v2) -- (v6) -- (v1);
		\draw (v1) edge[dashed] (v4);
		\draw (v2) edge[dashed] (v5);
		\draw (v3) edge[dashed] (v6);
	\end{tikzpicture}
	&
	\begin{tikzpicture}[line width=1pt, scale=1]
		\tikzstyle{every node}=[inner sep=0pt, minimum width=4.5pt] 
		\draw (0:1) node (v5) [draw, circle, fill=gray] {};
		\draw (60:1.1) node (v4) [draw, circle, fill=gray] {};
		\draw (120:1.1) node (v1) [draw, circle, fill=gray] {};
		\draw (180:1) node (v2) [draw, circle, fill=gray] {};
		\draw (240:1) node (v3) [draw, circle, fill=gray] {};
		\draw (300:1) node (v6) [draw, circle, fill=gray] {};
		\draw (90:0.29) node (v) [draw, circle, fill=gray] {};
		\draw (0:1.3) node () {\small $v_5$};
		\draw (60:1.35) node () {\small $v_4$};
		\draw (122:1.37) node () {\small $v_1$};
		\draw (180:1.3) node () {\small $v_2$};
		\draw (240:1.3) node () {\small $v_3$};
		\draw (300:1.3) node () {\small $v_6$};
		\draw (90:0.48) node () {\tiny $v$};
		\draw (v1) -- (v2) -- (v3) -- (v5) -- (v1) -- (v3) -- (v4) -- (v5) -- (v6) -- (v4) -- (v2) -- (v6) -- (v1);
		\draw (v) -- (v1); 
		\draw (v) -- (v3);
		\draw (v) -- (v4);
		\draw (v) -- (v6);
		\draw (v1) edge[dashed] (v4);
		\draw (v2) edge[dashed] (v5);
		\draw (v3) edge[dashed] (v6);
	\end{tikzpicture}
	&
	\begin{tikzpicture}[line width=1pt, scale=1]
		\tikzstyle{every node}=[inner sep=0pt, minimum width=4.5pt] 
		\draw (135:1)+(0,0.18) node (v1) [draw, circle, fill=gray] {};
		\draw (225:1)+(0,0.18) node (v2) [draw, circle, fill=gray] {};
		\draw (315:1)+(0,0.18) node (v3) [draw, circle, fill=gray] {};
		\draw (45:1)+(0,0.18) node (v4) [draw, circle, fill=gray] {};
		\draw (180:1.5)+(0,0.18) node (v5) [draw, circle, fill=gray] {};
		\draw (0:1.5)+(0,0.18) node (v6) [draw, circle, fill=gray] {};
		\draw (135:1.27)+(0,0.18) node () {\small $v_1$};
		\draw (225:1.27)+(0,0.18) node () {\small $v_2$};
		\draw (315:1.3)+(0,0.18) node () {\small $v_3$};
		\draw (45:1.25)+(0,0.18) node () {\small $v_4$};
		\draw (180:1.77)+(0,0.18) node () {\small $v_5$};
		\draw (0:1.8)+(0,0.18) node () {\small $v_6$};
		\draw (v1) -- (v2) -- (v3) -- (v4)-- (v1) -- (v3) -- (v2) -- (v4) -- (v6) -- (v2) -- (v5) -- (v1)-- (v3) -- (v6); 
	\end{tikzpicture}
	\\
	(a) $K_6\setminus M_3$
	&
	(b) $G_1$
	&
	(c) $G_2$
	\\
\end{tabular}
\end{center}
\caption{(a) $K_6\setminus M_3$, (b) $G_1$, (c) $G_2$.}
\label{fig1}
\end{figure}

If the digraph is not connected, then we can express its critical ideals as a function of the critical ideals of its connected components. 

\begin{Proposition}\label{disjoint}
If $G$ and $H$ are vertex disjoint digraphs, then 
\[
I_i(G\sqcup H,X)=\left\langle\, \bigcup_{j=0}^i I_j(G,X)I_{i-j}(H,X)\,\right\rangle \text{ for all } 1 \leq i \leq |V(G\sqcup H)|.
\]
\end{Proposition}
\begin{proof}
Let $Q=M\oplus N$ where $M\in M_m(\mathcal{P}[X])$ and $N\in M_n(\mathcal{P}[X])$, $k \in [m+n]$, and ${\bf r},{\bf s}\subseteq [m+n]$ with $|{\bf r}|=|{\bf s}|=k$.
It follows by using induction on $|{\bf r}_m|$, that
\[
\mathrm{det}(Q[{\bf r};{\bf s}])=
\begin{cases}
\mathrm{det}(Q[{\bf r}_m;{\bf s}_m]) \cdot \mathrm{det}(Q[{\bf r}_n;{\bf s}_n]) & \text{ if } |{\bf r}_m|=|{\bf s}_m|,\\
0 & \text{ otherwise},
\end{cases}
\]
where ${\bf r}_m={\bf r}\cap [m]$, ${\bf s}_m={\bf s}\cap [m]$, ${\bf r}_n={\bf r}\setminus {\bf r}_m$, and ${\bf s}_n={\bf s}\setminus {\bf s}_m$. 

Now, since $L(G\sqcup H,X)=L(G,X)\oplus L(H,X)$, we get that 
${\rm minors}_i(L(G\sqcup H,X))\setminus 0 \subseteq \{ m_1\cdot m_2 \, | \, m_1 \in {\rm minors}_j(L(G,X)) 
\text{ and } m_2\in {\rm minors}_{i-j}(L(H,X)) \text{ for some } 0\leq j \leq i\}$
for all $1 \leq i \leq |V(G\cup H)|$ and the result is obtained.
\end{proof}

Let $T_v$ be the trivial graph composed by the vertex $v$.
Since $I_1(T_v)=\langle x_v\rangle$, applying Proposition~\ref{disjoint}, we get the critical ideals of the trivial graph with $n$ vertices.

\begin{Corollary}\label{trivialgraph}
If $n\geq 1$ and $T_n$ is the graph with $n$ isolated vertices, then
\[
I_i(T_n)=\big\langle \{\prod_{j\in J} x_j\, \big| \, |J|=i\}\big\rangle \text{ for all } 1\leq i\leq n.
\]
\end{Corollary}

Now, we will establish some basic relationships between the critical ideals and the critical group.
Before doing this, we will introduce some notation.
Given a digraph $G$ with $n$ vertices and ${\bf d}\in \mathcal{P}^{V(G)}$, let $L(G, {\bf d})$ be the matrix obtained from $L(G, X)$
where we put $x_v={\bf d}_v$ for all $v\in V(G)$.
Also, for all $1\leq i \leq n$, let 
\[
I_i(G, {\bf d})=\{f({\bf d}) \, | \, f\in I_i(G,X)\}  \subseteq \mathbb{Z}.
\]
Given an induced subdigraph $H$ of $G$, the {\it degree vector} of $H$ in $G$ is given by $d_G(H)_v=\mathrm{deg}^+_G(v)$ for all $v\in V(H)$.

\begin{Proposition}\label{correspondence0}
Let $\mathcal{P}=\mathbb{Z}$ and $G$ be a digraph (possibly with multiple edges) with $n$ vertices. 
If $K(G)\cong \bigoplus_{i=1}^{n-1}\mathbb{Z}_{f_i}$ with $ f_1 | \cdots |f_{n-1}$, then
\[
I_i(G, d_G(G)) = \langle \prod_{j=1}^i f_j\rangle \text{ for all } 1\leq i \leq n-1.
\]
\end{Proposition}
\begin{proof}
Clearly $L(G,d_G(G))$ is equal to the Laplacian matrix $L(G)$ of $G$.
Thus ${\rm minors}_i(L(G,d_G(G)))={\rm minors}_i(L(G))$ for all $1\leq i \leq n$
and therefore
\[
I_i(G, d_G(G)) = \langle {\rm minors}_i(L(G,d_G(G))) \rangle= \langle {\rm minors}_i(L(G))\rangle= \langle \prod_{j=1}^i f_j\rangle \text{ for all } 1\leq i \leq n-1.
\vspace{-10mm}
\]
\end{proof}

\medskip

On the other hand, if $v$ is a vertex of $G$ and $L(G,v)$ is the {\it reduced Laplacian matrix}, the matrix obtained
from $L(G)$ by removing the row and column $v$, then we have the following strong version of Proposition~\ref{correspondence0}:

\begin{Proposition}\label{correspondence}
Let $\mathcal{P}=\mathbb{Z}$, $G$ be a connected digraph (possibly with multiple edges) with $n$ vertices, and $v$ a vertex of $G$.
If $G$ is Eulerian (that is, $d^+_{G}(v)=d^-_{G}(v)$ for all $v\in V(G)$) and $K(G)\cong \bigoplus_{i=1}^{n-1}\mathbb{Z}_{f_i}$ with $ f_1 | \cdots |f_{n-1}$, then
\[
I_i(G\setminus v, d_G(G\setminus v)) = \langle \prod_{j=1}^i f_j\rangle \text{ for all } 1\leq i \leq n-1.
\]
\end{Proposition}
\begin{proof}
Since $L(G\setminus v, d_G(G\setminus v))=L(G,v)$ (the reduced Laplacian matrix of $G$), 
by Proposition~\ref{correspondence0} we only need to prove that $K(G)\cong \mathbb{Z}^{V(G)\setminus v}/{\rm Im}\, L(G,v)^t$.
Now, if $I_{n-1,n-1}$ is the identity matrix of size $n-1$, then
\[
\left[\begin{array}{cc}
1&{\bf 1}\\
{\bf 0}&I_{n-1,n-1}
\end{array}\right]
L(G)
\left[\begin{array}{cc}
1&{\bf 1}\\
{\bf 0}&I_{n-1,n-1}
\end{array}\right]^t
= 
0\oplus L(G,v).
\]
Since ${\rm det} (\left[\begin{array}{cc}
1&{\bf 1}\\
{\bf 0}&I_{n-1,n-1}
\end{array}\right])=1$, then $L(G) \sim 0\oplus L(G,v)$ and we get the result.
\end{proof}

\medskip

\begin{Remark}
Note that, in general, Proposition~\ref{correspondence} is not valid for digraphs.
However, we can get a similar result for matrices in $M_{n\times n}$  with entries in a principal ideal domain
and such that $M{\bf 1}={\bf 0}$ and ${\bf 1}M={\bf 0}$.
\end{Remark}

The next example shows how Proposition~\ref{correspondence} can be used to recover the critical group of a graph from its critical ideals.
In this sense, critical ideals generalize the critical group of a graph.

\begin{Example}\label{ecorrespondence}
Let $H$ be the complete graph with six vertices minus a perfect matching as in Figure~\ref{fig1}(a).
and $G$ be a graph such that $H=G \setminus v$ for some vertex $v$ of $G$.
Thus, applying Proposition~\ref{correspondence} we can get the critical group of $G$ as an evaluation of the critical ideals of $H$.
For instance, if $G_1$ is the graph obtained from $H$ (see Figure~\ref{fig1}(b)) by adding a new vertex $v$ and the edges 
$vv_1,vv_3,vv_4,vv_6$, then $d_G(H)=(5,4,5,5,4,5)$.
Moreover, using the critical ideals of $H$ calculated in Example~\ref{ejemplo1}, 
we get that $f_i=1$ for all $i\leq 4$, $f_5=20$, $f_6=140$; that is, 
\[
K(G_1)\cong \mathbb{Z}_{20}\oplus\mathbb{Z}_{140}.
\]
On the other hand, if we only know the critical ideals of induced subgraphs of $G$ that are different to $G\setminus v$, 
then we cannot reconstruct completely the critical group of $G$. 
For instance, if $G_2$ is the graph obtained by adding the vertices $v_5$, $v_6$ 
and the edges $v_5v_1, v_5v_2, v_6v_2,v_6v_3,v_6v_4$ to the complete graph with four vertices (see Figure~\ref{fig1}(c)),
then it is not difficult to see using any algebraic system that $f_i(G_2)=1$ for $1\leq i\leq 4$ and $f_5(G_2)=185$.
However, when we apply Proposition~\ref{correspondence} to the critical ideals of 
the induced subgraph by the vertices $v_1$, $v_2$, $v_3$, and $v_4$ of $G_2$ (isomorphic to $K_4$)  
we can only obtain that $f_1(G_2)=f_2(G_2)=1$, $f_3(G_2)\, | \, 5$ and $f_4(G_2)\, | \,175$.
\end{Example}


\subsection{The invariant $\gamma$}\label{gamma}
In this subsection, we will present an invariant that will play an important role 
on the study of the critical ideals of a digraph.

\begin{Definition}
Given a digraph $G$ and a commutative ring with identity $\mathcal{P}$, let 
\[
\gamma_{\mathcal{P}}(G)=\max\{i\, | \, I_i(G, X)=\langle 1\rangle \}.
\]
\end{Definition}

Using the canonical homomorphisms $f:\mathbb{Z}\rightarrow \mathcal{P}$ given by $f(a)=af(1)$
is not difficult to get that $\gamma_{\mathbb{Z}}(G) \leq \gamma_{\mathcal{P}}(G)$.
For instance, it is clear that $\gamma_{\mathbb{Z}}(G) \leq \gamma_{\mathbb{Q}}(G)$.
Also, there exists a close relation between $\gamma_{\mathbb{Z}}(G)$ and the 
number of invariant factors of the critical group of $G$ that are equal to $1$.
For instance, if $IF_1(G)$ denote the number of invariant factors of the critical group of $G$ that are equal to $1$,
then $\gamma_{\mathbb{Z}}(G) \leq IF_1(G)$.
We found that $\gamma_\mathcal{P}(G)$ behaves better than the number of invariant factors of the critical group of a digraph that are one.
For instance, it is not difficult to see from the definition and Proposition~\ref{BasicProp} that if $H$ 
is an induced subdigraph of $G$, then $\gamma_\mathcal{P}(H)\leq \gamma_\mathcal{P}(G)$.
However, if $n\geq 3$, $G=K_{2,n}$ is a complete bipartite graph, and $H=K_{1,n}$ is an induced subgraph of $G$, then $IF_1(G)=2< n=IF_1(H)$. 

\medskip

Now, we present a relation between $\gamma_\mathcal{P}(G)$ and the stability and the clique numbers of $G$.
Before doing this, we will define the stability and the clique numbers of a graph.
A subset $S$ of the vertices of a graph $G$ is called {\it stable} or {\it independent} if there is no edge of $G$ with ends in $S$.
A stable set is called {\it maximal} if it is under the inclusion of sets.
The {\it stability number} of $G$, denoted by $\alpha(G)$, is given by
\[
\alpha(G)={\rm max}\{|S|\, | \, S \text{ is a stable set of } G\}.
\]
In a similar way, a subset $C$ of the vertices of a graph $G$ is called a {\it clique} if all the pairs of vertices in $C$ are joined by an edge of $G$.
A clique set is called {\it maximal} if it is under the inclusion of sets.
The {\it clique number} of $G$, denoted by $\omega(G)$, is given by
\[
\omega(G)={\rm max}\{|C|\, | \, C \text{ is a clique set of } G\}.
\]

\begin{Lemma}\label{bound1}
If $G$ is a digraph (possibly with multiple edges) and $v$ is a vertex of $G$, then 
\[
\gamma_\mathcal{P}(G)-\gamma_\mathcal{P}(G\setminus v)\leq 2.
\]
\end{Lemma}
\begin{proof}
We begin with a simple relation between the critical ideals of $G$ and $G\setminus v$. 
\begin{Claim}\label{claim:idealescriticos}
If $G$ is a digraph (possibly with multiple edges) with $V(G)=\{v_1,\ldots, v_n\}$, then 
\[
I_j(G,X)\subseteq \langle x_1I_{j-1}(G\setminus v_1, X\setminus x_1), I_{j-2}(G\setminus v_1, X\setminus x_1) \rangle \text{ for all } 1\leq j \leq n.
\]
\end{Claim}
\begin{proof}
Let $I=\{ i_1, \ldots i_j\}\subseteq [n]$, $I'=\{i'_i, \ldots, i'_j\} \subseteq [n]$ and $m_{I,I'}={\rm det}(L(G, X)[I,I'])\in I_j(G)$.
If $1\notin I\cup I'$, then $m_{I,I'}\in I_j(G\setminus v_1, X\setminus x_1)$.
In a similar way, if $1\in I\Delta I'$, then $m_{I,I'}\in I_{j-1}(G\setminus v_1, X\setminus x_1)$.
On the other hand, if $1\in I\cap I'$, then 
\[
m_{I,I'}\in \langle x_1 I_{j-1}(G\setminus v_1, X\setminus x_1), I_{j-2}(G\setminus v_1, X\setminus x_1) \rangle.
\]
Finally, the result follows because $I_j(G\setminus v_1, X\setminus x_1)\subseteq I_{j-1}(G\setminus v_1, X\setminus x_1)\subseteq I_{j-2}(G\setminus v_1, X\setminus x_1)$.
\end{proof}
Let $g=\gamma_\mathcal{P}(G\setminus v_1)$.
Since $I_{i}(G\setminus v_1, X\setminus x_1)\neq \langle1\rangle$ for all $i\geq g+1$, by Claim~\ref{claim:idealescriticos} 
\[
I_{g+3}(G, X) \subseteq \langle  x_1I_{g+2}(G\setminus v_1, X\setminus x_1), I_{g+1}(G\setminus v_1, X\setminus x_1) \rangle \neq \langle 1\rangle.
\]
Therefore $I_{g+3}(G, X) \neq \langle 1\rangle$.
That is, $\gamma_\mathcal{P}(G)-\gamma_\mathcal{P}(G\setminus v_1)\leq 2$.
\end{proof}

Since $\gamma_\mathcal{P}(T_m)=0$ and $\gamma_\mathcal{P}(K_{n+1})=1$ for all $m,n\ge 1$, then using Lemma~\ref{bound1} we have the following result:

\begin{Theorem}\label{bounds}
If $G$ is a digraph (possibly with multiple edges) with $n$ vertices, then
\[
\gamma_\mathcal{P}(G)\leq 2(n-\omega(G))+1\text{ and } \gamma_\mathcal{P}(G)\leq 2(n-\alpha(G)).
\]
\end{Theorem}
\begin{proof}
The result follows by using that $\gamma_\mathcal{P}(T_{\alpha(G)})=0$ (Corollary~\ref{trivialgraph}), $\gamma_\mathcal{P}(K_{\omega(G)})=1$ (Theorem~\ref{CK})
and the fact that $\gamma_\mathcal{P}(G)-\gamma_\mathcal{P}(G\setminus v)\leq 2$ for all $v\in V(G)$ (Lemma~\ref{bound1}).
\end{proof}

This result is interesting when the stability or the clique number is almost the number of vertices of the graph.
For instance, if $G$ is the complete graph ($\omega(G)=n$), then $\gamma_\mathcal{P}(G)\leq 1$.
Theorem~\ref{CK} proves that this bound is tight.
Similarly to Theorem~\ref{bounds}, in~\cite{lorenzini91} an upper bound for the number of invariant factors different to one 
of the critical group of a graph $G$ in terms of the number of independent cycles of $G$ was found.

In~\cite{lorenzini89} a similar result to the obtained in the proof of Theorem~\ref{bounds} was obtained. 
Namely in~\cite{lorenzini89} it was shown that if $G$ is a simple graph and $e\in E(G)$, 
then the number of invariant factors different to $1$ of $G$ and $G\setminus e$ differ by at most $1$.

Clearly, a simple graph $G$ has $\gamma_\mathcal{P}(G)=0$ if and only if $G$ is the trivial graph.
Moreover, in~\cite{alfaro12} it was shown that a simple graph $G$ has $\gamma_\mathbb{Z}(G)=1$ if and only if $G$ is the complete graph.
Also, all the simple graphs with $\gamma_\mathbb{Z}$ equal to $2$ were characterized in~\cite{alfaro12}.
 
It is not difficult to prove that the bound $\gamma_\mathcal{P}(G)\leq 2(n-\alpha(G))$ is tight. 
For instance, it is easy to prove that if $P_{2n+1}$ 
is an odd path (see Corollary~\ref{coropath}), then $\alpha(P_{2n+1})=n+1$ and $\gamma_\mathcal{P}(P_{2n+1})=2n=2(2n+1-(n+1))$. 
Moreover, in~\cite{trees} it was shown that if $T$ is a tree, then $\gamma_\mathbb{Z}(T)$ is equal to its $2$-matching number.
This result proves that the bound $\gamma_\mathcal{P}(G)\leq 2(n-\alpha(G))$ 
is tight for any value of the stability number and the number of vertices of the graph.
An interesting open question is the characterization of the simple graphs that satisfy the bounds given in Theorem~\ref{bounds}.

Next example shows a graph $G$ with $\gamma_\mathbb{Z}(G)=5$ such that $L(G,X)$ has no $5$-minor equal to $1$.

\begin{Example}
Let $G$ be the cone of $H$ (obtained from $H$ when we add a new vertex $v$ and all the edges between the vertex $v$ and the vertices of $H$), see Figure~\ref{fige}.
\begin{figure}[h]
\begin{center}
\begin{tabular}{c@{\extracolsep{2cm}}c}
\multirow{9}{3cm}{
	\begin{tikzpicture}[line width=1pt, scale=1]
		\tikzstyle{every node}=[inner sep=0pt, minimum width=4.5pt] 
		\draw (135:1) node (v1) [draw, circle, fill=gray] {};
		\draw (225:1)+(-1,0) node (v2) [draw, circle, fill=gray] {};
		\draw (225:1) node (v3) [draw, circle, fill=gray] {};
		\draw (315:1) node (v4) [draw, circle, fill=gray] {};
		\draw (315:1)+(1,0) node (v5) [draw, circle, fill=gray] {};
		\draw (45:1) node (v6) [draw, circle, fill=gray] {};
		\draw (v1)+(-0.2,0.2) node () {\small $v_1$};
		\draw (v2)+(0,-0.3) node () {\small $v_2$};
		\draw (v3)+(0,-0.3) node () {\small $v_3$};
		\draw (v4)+(0,-0.3) node () {\small $v_4$};
		\draw (v5)+(0,-0.3) node () {\small $v_5$};
		\draw (v6)+(0.2,0.2) node () {\small $v_6$};
		\draw (0,1) node () {\small $H$};
		\draw (v1) -- (v3) -- (v4) -- (v6) -- (v1);
		\draw (v1) -- (v2) -- (v3);
		\draw (v6) -- (v5) -- (v4);
		\draw (v6) edge (v3);
	\end{tikzpicture}
}
& \\
&
$
L(c(H), X)=
\left[\begin{array}{ccccccc}
 x_1 & -1 &  -1 &  0 & 0 & -1 & -1\\
 -1 &  x_2 & -1 & 0 & 0 & 0 & -1\\
 -1 & -1 &  x_3 & -1 & 0 & -1 & -1\\
0 &  0 &  -1 &  x_4 & -1 & -1 & -1\\
0  & 0 &  0 &  -1 & x_5 & -1 & -1\\
-1 & 0 &  -1 &  -1 & -1 & x_6 & -1\\
-1 & -1 &  -1 &  -1 & -1 & -1 & x_7
\end{array}\right]
$\\
& \\
\end{tabular}
\end{center}
\caption{A graph $H$ with six vertices and the generalized Laplacian matrix of its cone.}
\label{fige}
\end{figure}
Since ${\rm det}(L(G,X)[\{1, 2, 3,4, 5\}, \{2, 3, 5, 6, 7\}]) =x_2+x_5+x_2x_5$ and 
${\rm det}(L(G,X)[\{1, 2, 3, 5, 6\}, \{ 2,4, 5, 6, 7\}])=-(1+x_2+x_5+x_2x_5)$, then $\gamma_\mathbb{Z}(G)=5$. 
However, it is not difficult to check that no $5$-minor of $L(G,X)$ is equal to one or another integer.
\end{Example}

Now, we turn our attention to the critical ideals of the complete graph.


\subsection{The critical ideals of the complete graphs}\label{complete}
We begin this subsection with an expression for the determinant of the complete graph.

\begin{Proposition}\label{DK}
If $K_n$ is the complete graph with $n\geq 1$ vertices, then 
\[
\mathrm{det}(L(K_n,X))=\prod_{j=1}^{n} (x_j+1) - \sum_{i=1}^n \prod_{j\neq i} (x_j+1).
\]
\end{Proposition}
\begin{proof}
We will use induction on $n$.
For $n=1$, it is clear that $\mathrm{det}(L(K_n,X))=x_1=(x_1+1)-1$.
Now, assume that $n\geq 2$.
Expanding the determinant of $L(K_{n+1},X)$ by the last column and using induction hypothesis
\vspace{-3mm}
\begin{eqnarray*}
\mathrm{det}(L(K_{n+1},X)) &=& x_{n+1}\cdot \mathrm{det}(L(K_n,X))+ \sum_{k=1}^n \mathrm{det}(L(K_n,X))_{x_k=-1}\\
&=& x_{n+1}\cdot \left[ \prod_{j=1}^{n} (x_j+1) - \sum_{i=1}^n \prod_{j\neq i} (x_j+1) \right]-\sum_{k=1}^n \prod_{j\neq k} (x_j+1)\\
&=&\prod_{j=1}^{n+1} (x_j+1) -\prod_{j=1}^{n} (x_j+1)  - (x_{n+1}+1)\sum_{i=1}^n \prod_{j\neq i} (x_j+1)\\
&=& \prod_{j=1}^{n+1} (x_j+1) - \sum_{i=1}^{n+1} \prod_{j\neq i} (x_j+1).
\end{eqnarray*}
\end{proof}

The next result gives us a description of a reduced Gr\"obner basis of the critical ideals of the complete graph.

\begin{Theorem}\label{CK}
If $K_n$ is the complete graph with $n\geq 2$ vertices and $1\leq m \leq n-1$, then 
\[
 B_m=\big\{\prod_{i\in I} (x_i+1) \, |\, I\subseteq [n] \text{ and } |I|=m-1\big\}
\] 
is a reduced Gr\"obner basis of $I_m(K_n, X)$ with respect to the graded lexicographic order.
\end{Theorem}
\begin{proof}
First, we will prove that $B_m$ generates $I_m(K_n, X)$.
If $M$ is a square submatrix of $L(K_n,X)$ of size $m$, 
then there exist $I\subset [n]$ with $|I|=m$ and $J\subseteq I$ such that $M$ is equal to 
\[
L(K_n,X)[I; I]_{\{x_j=-1 \text{ for all }j\in J\}}.
\]
Thus, since
\[
\mathrm{det}(L(K_n,X))[I;I]_{\{x_j=-1 \text{ for all }j\in J\}}=
\begin{cases}
\prod_{i_j\in I\setminus J} (x_{i_j}+1) &\text{  if } |J|=1,\\
0 &\text{  if } |J|\geq 2,
\end{cases}
\]
and $\mathrm{det}(L(K_n,X))[I;I]=-(x_{i_m}+1)\cdot \mathrm{det}(L(K_n,X))[I;I]_{\{x_{i_m}=-1\}} + \sum_{k=1}^m \mathrm{det}(L(K_n,X))[I;I]_{\{x_{i_k}=-1\}}$,
then $B_m$ generates $I_m(K_n,X)$.

Finally, we will prove that $B_m$ is a reduced Gr\"obner basis of $I_m(K_n, X)$ for all $1\leq m \leq n-1$.
Let $I_1,I_2\subset [n]$ with $|I_1|,|I_2|=m-1$ and $p_I=\prod_{i\in I} (x_i+1)$. 
It is not difficult to see that $lt(p_I)=\prod_{i\in I} x_i$ and
\[
p_{I_1}\cdot \prod_{i\in I_2\setminus I_1} (x_i+1)-p_{I_2} \cdot \prod_{i\in I_1\setminus I_2} (x_i+1)=0.
\]
Thus
\begin{eqnarray*}
S(p_{I_1}, p_{I_2}) &=& \frac{lt(p_{I_2})}{lt(p_{I_1\cap I_2})} p_{I_1}- \frac{lt(p_{I_1})}{lt(p_{I_1\cap I_2})} p_{I_2}\\  
&=& \left(\prod_{i\in I_2\setminus I_1} (x_i+1)-\prod_{i\in I_2\setminus I_1}  x_i\right) \cdot  p_{I_2} - 
\left(\prod_{i\in I_1\setminus I_2} (x_i+1)-\prod_{i\in I_1\setminus I_2} x_i\right)\cdot p_{I_1}
\end{eqnarray*}
and $S(p_{I_1}, p_{I_2})\rightarrow_{B_m} 0$.
Therefore, $B_m$ is a reduced Gr\"obner basis of $I_i(K_n, X)$.
\end{proof}

\begin{Remark}
Note that, since $\prod_{\emptyset}=1$, $I_{1}(K_n, X)=\langle 1\rangle$ and therefore $\gamma(K_n)=1$.
\end{Remark}

Using the expression for the critical ideals of the complete graph given in Theorem~\ref{CK}, 
we can get the primary decomposition of the critical ideals of the complete graph.
\begin{Corollary}
If $K_n$ is the complete graph with $n\geq 2$ vertices and $1\leq m \leq n-1$, then 
\[
I_m(K_n, X)=\bigcap_{I\subset [n], |I|=n-m+2} \langle \{x_i+1\, | \, i\in I\}\rangle.
\]
\end{Corollary}

As an application of Theorem~\ref{CK}, we find the critical group of all the graphs with $\omega(G)=|V(G)|-1$.

\begin{Corollary}
Let $n> m\geq 1$ and $K_{n+1} \setminus S_m$ be the graph obtained from $K_{n+1}$ by deleting the $m$ edges of the star $S_m$.
Then
\[
K(K_{n+1} \setminus S_m)=
\begin{cases}
\mathbb{Z}_{n+1}^{n-2m}\oplus\mathbb{Z}_{n(n+1)}^{m-2}\oplus\mathbb{Z}_{n(n+1)(n-m)} & \text{ if } m\leq \left\lfloor n/2\right\rfloor,\\
\\
\mathbb{Z}_{n}^{2m-n}\oplus\mathbb{Z}_{n(n+1)}^{n-m-2}\oplus\mathbb{Z}_{n(n+1)(n-m)} & \text{ if } m\geq \left\lfloor n/2\right\rfloor.
\end{cases}
\]
\end{Corollary}
\begin{proof}
Let $v_{n+1}$ be the vertex of $K_{n+1}\setminus S_m$ of degree $n-m$. 
Thus
\[
{\bf d}=d_{K_{n+1}-S_m}(K_{n+1}\setminus v_{n+1})=(\underbrace{n-1,\ldots,n-1}_{m \mathrm{\ times}},\underbrace{n,\ldots,n}_{n-m\mathrm{\ times}}).
\]
Now, by Proposition~\ref{DK}, $I_n(K_n,{\bf d})= \langle n^{m-1}(n+1)^{n-m-1}(n-m)\rangle$,
\[
I_{i+1}(K_n,{\bf d})=
\begin{cases}
\langle 1\rangle &\text{ if } i\leq m,\\
\langle (n+1)^{i-m}\rangle &\text{ if } m<i\leq n-m, \\
\langle n^{i+m-n}(n+1)^{i-m}\rangle &\text{ if } n-m<i< n-1,
\end{cases}
\]
when $m\leq \left\lfloor n/2\right\rfloor$, and
\[
I_{i+1}(K_n,{\bf d})=
\begin{cases}
\langle 1\rangle &\text{ if } i\leq n-m, \\
\langle n^{i+m-n}\rangle&\text{ if } n-m<i\leq m, \\
\langle n^{i+m-n}(n+1)^{i-m}\rangle &\text{ if } m<i<n-1,
\end{cases}
\]
when $m\geq \left\lfloor n/2\right\rfloor$.
Finally, applying Proposition~\ref{correspondence} we get the result.
\end{proof}

To the authors knowledge, the critical group of these graphs had not been calculated before.
We finish this section by exploring a relation between the critical ideals and the characteristic polynomial of the adjacency matrix of a graph.


\subsection{The critical ideals and the characteristic polynomials}\label{characteristic}
In~\cite{lorenzini08}, Lorenzini showed a deep relation between the critical group and the Laplacian spectrum of a graph.
For instance, Lorenzini (\cite[Proposition 3.2]{lorenzini08}) proved that if $\lambda$ is an integer eigenvalue of $L(G)$  of multiplicity $\mu(\lambda)$,
then $K(G)$ contains a subgroup isomorphic to $\mathbb{Z}_{\lambda}^{\mu(\lambda)-1}$. 
In this subsection, we will present the relation that exists between the critical ideals 
and the characteristic polynomial of the adjacency matrix of a graph.
If $G$ is loopless and we take $x_i=t$ for all $1\leq i \leq n$, then $\mathrm{det}(L(G,X))$
is equal to the characteristic polynomial $p_G(t)$ of the adjacency matrix of $G$
and the critical ideals of $G$ are ideals in $\mathcal{P}[t]$.
Let $I_i(G,t)=I_i(G, X)_{\{x_j=t\, | \, \forall \, 1 \leq j \leq n\}}$ for all $1\leq i\leq n$.
It is not difficult to see that if $\mathcal{P}$ is a field, then the ideals $I_i(G,t)$ are principal.
That is, there exist $p_i(t)\in \mathcal{P}[t]$ for all $1\leq i \leq n$ such that
\[
I_i(G,t)=\langle \prod_{j=1}^i p_j(t)\rangle \text{ for all } 1\leq i\leq n.
\]
Thus, $p_G(t)=\prod_{j=1}^n p_j(t)$ is a factorization of the characteristic polynomial of the adjacency matrix of $G$. 
In a similar way, if we take $x_i=d_i-t$ for all $1\leq i \leq n$, then
we recover a factorization of the characteristic polynomial  of the Laplacian matrix of $G$ from their critical ideals.

Therefore, the critical ideals of a graph are a generalization of the characteristic polynomial the adjacency matrix and the Laplacian matrix of $G$. 
For instance, if $G$ is the complete graph with six vertices minus a perfect matching (as in Example~\ref{ejemplo1}) 
and $\mathcal{P}=\mathbb{Q}$, then
\[
I_i(G,t)=
\begin{cases}
\langle 1\rangle & \text{ if } 1\leq i\leq 4,\\
\langle t^2(t+2)\rangle & \text{ if } i= 5,\\
\langle t^3(t+2)^2(t-4)\rangle & \text{ if } i=6.
\end{cases}
\]
Note that in general the $t$-critical ideals depend on the base ring. 
For instance, if  $\mathcal{P}=\mathbb{Z}$, then
\[
I_i(G,t)=
\begin{cases}
\langle 1\rangle & \text{ if } i=1,2,\\
\langle 2,t\rangle & \text{ if } i= 3,\\
\langle t^2,4t\rangle & \text{ if } i= 4,\\
\langle t^3(t+2),4t^2(t+2)\rangle & \text{ if } i= 5,\\
\langle t^3(t+2)^2(t-4)\rangle & \text{ if } i=6.
\end{cases}
\]

The critical ideals are a stronger invariant than the adjacency spectrum of the graph.
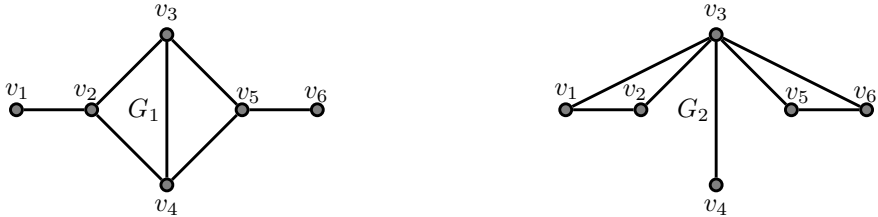
\begin{figure}[h] \centering
	\begin{tabular}{c@{\extracolsep{3cm}}c}
		\begin{tikzpicture}[line width=1.1pt, scale=1]
			\tikzstyle{every node}=[inner sep=0pt, minimum width=4.5pt] 
			\draw (180:2) node[draw, circle, fill=gray] (v1) {};
 			\draw (0:1) node[draw, circle, fill=gray] (v5) {};
 			\draw (90:1) node[draw, circle, fill=gray] (v3) {};
 			\draw (180:1) node[draw, circle, fill=gray] (v2) {};
 			\draw (270:1) node[draw, circle, fill=gray] (v4) {};
			\draw (0:2) node[draw, circle, fill=gray] (v6) {};
			\draw (v1) to (v2);
 			\draw (v2) to (v3);
 			\draw (v2) to (v4);
 			\draw (v3) to (v5);
 			\draw (v4) to (v5);
			\draw (v3) to (v4);
			\draw (v5) to (v6);
			\draw (173:2) node {\small $v_1$};
			\draw (168:1.1) node {\small $v_2$};
			\draw (90:1.3) node {\small $v_3$};
			\draw (270:1.3) node {\small $v_4$};
 			\draw (10:1.1) node {\small $v_5$};
			\draw (6:2) node {\small $v_6$};
			\draw (-0.3,0) node {\small $G_1$};
		\end{tikzpicture}
		&
		\begin{tikzpicture}[line width=1.1pt,scale=1]
			\tikzstyle{every node}=[inner sep=0pt, minimum width=4.5pt] 
			\draw (180:2) node[draw, circle, fill=gray] (v1) {};
 			\draw (0:1) node[draw, circle, fill=gray] (v5) {};
 			\draw (90:1) node[draw, circle, fill=gray] (v3) {};
 			\draw (180:1) node[draw, circle, fill=gray] (v2) {};
 			\draw (270:1) node[draw, circle, fill=gray] (v4) {};
			\draw (0:2) node[draw, circle, fill=gray] (v6) {};
			\draw (v1) to (v2);
 			\draw (v1) to (v3);
 			\draw (v2) to (v3);
 			\draw (v3) to (v4);
 			\draw (v3) to (v5);
			\draw (v3) to (v6);
			\draw (v5) to (v6);
			\draw (173:2) node {\small $v_1$};
			\draw (168:1.1) node {\small $v_2$};
			\draw (90:1.3) node {\small $v_3$};
			\draw (270:1.3) node {\small $v_4$};
 			\draw (10:1.1) node {\small $v_5$};
			\draw (6:2) node {\small $v_6$};
			\draw (-0.3,0) node {\small $G_2$};
		\end{tikzpicture}
	\end{tabular}
	\caption{\small The graphs $G_1$ and $G_2$ have the same characteristic polynomial.}
	\label{Gfullh}
\end{figure}
For instance, if $\mathcal{P}=\mathbb{Z}$, $G_1$ and $G_2$ are the graphs given in Figure~\ref{Gfullh}, 
then they are cospectral, but 
\[
I_i(G_1,t)=
\begin{cases}
\langle 1\rangle & \text{ if } 1\leq i \leq 4,\\
\langle 2(t+1), (t+1)\cdot (t^2+1)\rangle & \text{ if } i=5,\\
\langle (t-1)\cdot (t+1)^2\cdot (t^3-t^2-5t+1)\rangle & \text{ if } i=6,
\end{cases}
\]
and
\[
I_i(G_2,t)=
\begin{cases}
\langle 1\rangle & \text{ if } 1\leq i \leq 3,\\
\langle 2,(t+1)\rangle & \text{ if } i=4,\\
\langle 4(t+1), (t+1)\cdot (t-3)\rangle & \text{ if } i=5,\\
\langle (t-1)\cdot (t+1)^2\cdot (t^3-t^2-5t+1)\rangle & \text{ if } i=6.
\end{cases}
\]


\section{Critical ideals of the cycle}\label{combinatorial}
The main goal of this section is to get a minimal set of generators and the reduced Gr\"obner basis of the critical ideals of the cycle.
Before of doing this, we will get some combinatorial expressions for the minors of the generalized Laplacian matrix of a digraph.
These expressions will be very useful in order to get some algebraic relations between the generators of the critical ideals of the cycle.

We begin with some concepts of digraphs.
Given a digraph $D$, a subdigraph $C$ of $D$ is called a {\it directed} \emph{$1$-factor} if and only if $d_C^+(v)=d_C^-(v)=1$ for all $v\in V(C)$.
The number of connected components of $C$ will be denoted by $c(C)$.

\begin{Theorem}\label{det}
If $D$ is a digraph with $n$ vertices (possibly with multiple arcs and loops), then 
\[
\mathrm{det}(-A(D))=\sum_{C\in \overset{\rightarrow}{\mathcal{F}}} (-1)^{c(C)},
\]
where $\overset{\rightarrow}{\mathcal{F}}$ is the set of spanning directed $1$-factors of $D$.
Moreover,
\[
\mathrm{det}(L(D,X))=\sum_{U\subseteq V(D)} \mathrm{det}(-A(D[U]^{wl})) \cdot \prod_{v\notin U} x_v,
\]
where $D[U]$ is the induced subgraph of $D$ by $U$ and $D[U]^{wl}$ is the graph obtained from $D[U]$ when we delete the possible loops.
\end{Theorem}
\begin{proof}
It follows from arguments similar to those in~\cite[Proposition 7.2]{biggs93}.
\end{proof}

\begin{Remark}
If we identify a graph $G$ with the digraph $D_G$, obtained from $G$ 
by replacing each edge $uv$ of $G$ by the arcs $\overset{\longrightarrow}{uv}$ and $\overset{\longrightarrow}{vu}$,
then an elementary graph is identified with a spanning directed $1$-factor and therefore the first part of
Theorem~\ref{det} is equivalent to~\cite[Proposition 7.2]{biggs93}.
In this way, Theorem~\ref{det} generalizes the expression obtained in~\cite[Proposition 7.2]{biggs93}.
\end{Remark}

Next we will present an example of the use of Theorem~\ref{det}.

\begin{Example}\label{exa1}
If $D$ is the digraph given by $V(D)=\{v_1,v_2,v_3,v_4\}$ (see Figure~\ref{fig3}) and 
\[
E(D)=\{\overset{\longrightarrow}{v_1v_1},\overset{\longrightarrow}{v_1v_2}, \overset{\longrightarrow}{v_2v_3},\overset{\longrightarrow}{v_3v_4},\overset{\longrightarrow}{v_4v_1},\overset{\longrightarrow}{v_4v_2}\},
\]
then the digraph $D$ has two spanning directed $1$-factors;
$\{\overset{\longrightarrow}{v_1v_2}, \overset{\longrightarrow}{v_2v_3},\overset{\longrightarrow}{v_3v_4},\overset{\longrightarrow}{v_4v_1}\}$ 
and $\{\overset{\longrightarrow}{v_1v_1},\overset{\longrightarrow}{v_2v_3},\overset{\longrightarrow}{v_3v_4},\overset{\longrightarrow}{v_4v_2}\}$.
Thus, $\mathrm{det}(A(D))=(-1)^1+(-1)^2=0$.
Also, since $\{\overset{\longrightarrow}{v_1v_2}, \overset{\longrightarrow}{v_2v_3},\overset{\longrightarrow}{v_3v_4},\overset{\longrightarrow}{v_4v_1}\}$ and 
$\{\overset{\longrightarrow}{v_2v_3},\overset{\longrightarrow}{v_3v_4},\overset{\longrightarrow}{v_4v_2}\}$ are the directed $1$-factors of $D^{wl}$,
$\mathrm{det}(L(D,X))=x_1x_2x_3x_4-x_1-1$.
\end{Example}

\vspace{-3mm}
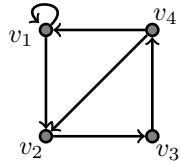
\begin{figure}[h]
\begin{center}
	\begin{tikzpicture}[line width=1pt, scale=1]
		\tikzstyle{every node}=[inner sep=0pt, minimum width=4.5pt] 
		\draw (135:1)+(0,0.18) node (v1) [draw, circle, fill=gray] {};
		\draw (225:1)+(0,0.18) node (v2) [draw, circle, fill=gray] {};
		\draw (315:1)+(0,0.18) node (v3) [draw, circle, fill=gray] {};
		\draw (45:1)+(0,0.18) node (v4) [draw, circle, fill=gray] {};
		\draw (150:1.20)+(0,0.18) node () {\small $v_1$};
		\draw (225:1.27)+(0,0.18) node () {\small $v_2$};
		\draw (315:1.3)+(0,0.18) node () {\small $v_3$};
		\draw (45:1.25)+(0,0.18) node () {\small $v_4$};
		\draw (v1) edge[->] (v2); 
		\draw (v1) edge[->, loop] (v1); 
		\draw (v2) edge[->] (v3);
		\draw (v3) edge[->] (v4);
		\draw (v4) edge[->] (v1) edge[->] (v2);
	\end{tikzpicture}
\end{center}
\caption{ A digraph with four vertices.}
\label{fig3}
\end{figure}

When $G$ is a tree, the determinant of the generalized Laplacian matrix given in Theorem~\ref{det}, only depends on its matchings.
Therefore, in this case we can get an explicit combinatorial expression for the 
determinant of the generalized Laplacian.
Given a set of edges $\mu$, let $V(\mu)$ be the set of vertices of the induced subgraph by $\mu$.

\begin{Lemma}\label{dettree1m}
If $T$ is a tree, then
\[
\mathrm{det}(L(T,X))=\sum_{\mu \in\mathcal{V}_1(T)} (-1)^{|\mu|} \! \! \prod_{v \notin V(\mu)}\!\! x_v.
\]
\end{Lemma}
\begin{proof}
It follows from Theorem~\ref{det} and the fact that if $S$ is a forest, then 
\[
\mathrm{det}(A(S))=
\begin{cases}
(-1)^{|\mu|} & \text{ if } \mu \text{ is a perfect matching of } S, \text{ and}\\
0 & \text{ otherwise.}
\end{cases}
\vspace{-5mm}
\]
\end{proof}

Moreover, Theorem~\ref{det} and Lemma~\ref{dettree1m} can be used to get combinatorial expressions 
for the determinants of the generalized Laplacian matrices of graphs.

\begin{Corollary}\label{detkpc}
Let $n$ be a natural number, $K_n$ be the complete graph with $n$ vertices, $P_n$ be the path with $n$ vertices, and $C_n$ the cycle with  $n$ vertices.
Then
\begin{description}
\item[(i)] $\displaystyle{ \mathrm{det}(L(K_n, X))=\sum_{I\subseteq [n]} (|I|-n+1)\cdot \prod_{i\in I} \, x_i }$,
\\
\item[(ii)] $\displaystyle{ \mathrm{det}(L(P_n, X))=\sum_{\mu \in \mathcal{V}_1(P_n)} (-1)^{|\mu|}\cdot \prod_{v\not\in V(\mu)} \, x_v }$,
\\
\item[(iii)] $\displaystyle{ \mathrm{det}(L(C_n, X))=\sum_{\mu \in \mathcal{V}_1(C_n)} (-1)^{|\mu|}\cdot \prod_{v\not\in V(\mu)}  x_v -2}$,
\end{description}
where $\mathcal{V}_1(G)$ is the set of matchings of $G$. 
\end{Corollary}
\begin{proof}
(i) Follows by Theorem~\ref{det} and the fact that $\mathrm{det}(A(K_m))=-m+1$ for all $m\geq 2$.
(ii) Follows directly from Lemma~\ref{dettree1m}.
(iii) Let $P_i^l$ be the induced path of $C_n$ of length $l$ that begins in the vertex $i$ of $C_n$.
Since $\mathcal{V}_1(C_n)=\mathcal{V}_1(C_n\setminus v_1v_n) \sqcup \{  \mu \in \mathcal{V}_1(C_n)\, | \, v_1v_n\in \mu\}$,
\[
\displaystyle{\mathrm{det}(L(C_n, X))=\mathrm{det}(L(P_1^n, X))-\mathrm{det}(L(P_2^{n-2}, X))-2
=\sum_{\mu \in \mathcal{V}_1(C_n)} (-1)^{|\mu|}\cdot \prod_{v\not\in V(\mu)}  x_v -2}.\vspace{-7mm}
\]
\end{proof}

\vspace{3mm}

We will show that every minor of a generalized Laplacian of a digraph is equal 
to an evaluation of the determinant of the generalized Laplacian of some digraph.

Given a digraph (or a graph) $D$ and $u\neq v$ two vertices of $D$, let $D(u;v)$
be the digraph obtained from $D$ by deleting the arcs leaving $u$ and entering $v$ 
(remember that each edge of a graph is considered as two arcs in both directions)
and identifying the vertices $u$ and $v$ in a new vertex, denoted by $u\circ v$.
Also, if $u=v$, then $D(u;v)$ is defined as $D\setminus u$.

\begin{Example}\label{path}
If $P_n$ is the path with $n$ vertices (simple graph), then 
\[
P_n(v_1;v_n)\cong C_{n-1}\setminus \{\overset{\longrightarrow}{v_1v_2},\overset{\longrightarrow}{v_{n-1}v_1}\}.
\]
\end{Example}

\medskip

Now, given two matrices $M,N \in M_{m\times m}(\mathcal{P})$, we say that $M$ and $N$ are {\it strongly equivalent}, 
denoted by $N\approx M$, if there exist $P$ and $Q$ permutation matrices such that $N=PMQ$.
Also, given $U=\{u_1,\ldots, u_s\}$ and $V=\{v_1,\ldots, v_s\}$ ordered subsets of $V(D)$ with $s\geq 2$ ,
we define $D(U;V)$ inductively as $D(U\setminus u_s; V\setminus v_s)(u_s;v_s)$.

\begin{Lemma}\label{minors}
If $D$ is a digraph and $U=\{u_1,\ldots, u_s\}$, $V=\{v_1,\ldots, v_s\}$ ordered subsets of $V(D)$, then
\[
L(D, X)(U;V)\approx L(D(U;V), X)_{\{x_{u_i\circ v_i}=-m_{(v_i,u_i)}\, | \, i=1,\ldots,s\}}.
\]
\end{Lemma}
\begin{proof}
It follows from the construction of $D(U;V)$.
\end{proof}

\begin{Remark}
Clearly, $D(U;V)$ and $L(D(U;V),X)$ depend on the order of the elements of the subsets $U$ and $V$.
However, if $\sigma\in S_s$ is a permutation, then
\[
L(D(U;V),X)_{\{x_{u_i\circ v_i}=-m_{(v_i,u_i)}\, |\, i=1,\ldots,s\}} \approx 
L(D(U;\{v_{\sigma(1)},\ldots, v_{\sigma(s)}\}),X)_{\{x_{u_i\circ v_{\sigma(i)}}=-m_{(v_{\sigma(i)},u_i)}\, |\, i=1,\ldots,s\}},
\]
that is, in some sense $L(D(U;V),X)_{\{x_{u_i\circ v_i}=-m_{(v_i,u_i)}\, | \, i=1,\ldots,s\}}$
does not depend on the order of the elements of the subsets $U$ and $V$.
\end{Remark}

\begin{Example}\label{exa2}
Let $D$ be the cycle with six vertices, $U=\{v_1,v_2\}$, and $V=\{v_6,v_5\}$.
Then we get that $D(U;V)$ is the digraph with four vertices $\{v_3,v_4,v_2\circ v_5, v_1\circ v_6\}$ (see Figure~\ref{fig4}(a))
and arcs $\{v_3v_4,v_4v_3, v_3(v_2\circ v_5),(v_1\circ v_6)v_4\}$.
On the other hand, if we change the order of the elements in $U$, then $D(U;V)$ is given by the graph in Figure~\ref{fig4}(b).
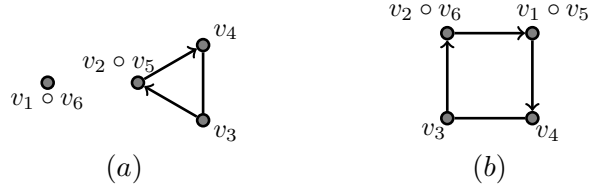
\begin{figure}[h]
\begin{center}
\begin{tabular}{c@{\extracolsep{20mm}}c}
	\begin{tikzpicture}[line width=1pt, scale=1]
		\tikzstyle{every node}=[inner sep=0pt, minimum width=4.5pt] 
		\draw (-1.2,0) node (v16) [draw, circle, fill=gray] {};
		\draw (0,0) node (v25) [draw, circle, fill=gray] {};
		\draw (-30:1) node (v3) [draw, circle, fill=gray] {};
		\draw (30:1) node (v4) [draw, circle, fill=gray] {};
		\draw (-1.2,-0.25) node () {\small $v_1\circ v_6$};
		\draw (-0.25,0.25) node () {\small $v_2\circ v_5$};
		\draw (-30:1.35) node () {\small $v_3$};
		\draw (30:1.35) node () {\small $v_4$};
		\draw (v3) -- (v4); 
		\draw (v3) edge[->] (v25);
		\draw (v25) edge[->] (v4);
	\end{tikzpicture}
&
	\begin{tikzpicture}[line width=1pt, scale=1]
		\tikzstyle{every node}=[inner sep=0pt, minimum width=4.5pt] 
		\draw (45:0.8) node (v15) [draw, circle, fill=gray] {};
		\draw (135:0.8) node (v26) [draw, circle, fill=gray] {};
		\draw (225:0.8) node (v3) [draw, circle, fill=gray] {};
		\draw (315:0.8) node (v4) [draw, circle, fill=gray] {};
		\draw (45:1.20) node () {\small $v_1\circ v_5$};
		\draw (135:1.2) node () {\small $v_2\circ v_6$};
		\draw (225:1.05) node () {\small $v_3$};
		\draw (315:1.09) node () {\small $v_4$};
		\draw (v3) edge[->] (v26);
		\draw (v26) edge[->] (v15);
		\draw (v15) edge[->] (v4);
		\draw (v3) -- (v4); 
	\end{tikzpicture}
\\
$(a)$
&
$(b)$
\end{tabular}		
\end{center}
\caption{Digraphs with four vertices.}
\label{fig4}
\end{figure}

\vspace{-3mm}
It is not difficult to see that 
{\scriptsize
\begin{eqnarray*}
L(D, X)(U;V)
=
\left[\begin{array}{cccc}
0& -1&x_3& -1\\
0& 0&-1& x_4\\
0& 0& 0& -1\\
-1& 0 &0 &0
\end{array}\right]
&\hspace{-5mm} \approx \hspace{-5mm}&
\left[\begin{array}{cccc}
x_3& -1& -1& 0\\
-1& x_4& 0& 0\\
 0& -1 & 0& 0\\
0 &0 &0& -1
\end{array}\right]
=L(D(\{v_1,v_2\};\{v_6,v_5\}), X)_{\{x_{v_2\circ v_5}=0, \, x_{v_1\circ v_6}=-1 \}}.\\
\\
& \hspace{-5mm}\approx \hspace{-5mm}&
\left[\begin{array}{cccc}
x_3& -1& -1& 0\\
-1& x_4& 0& 0\\
 0& 0 & 0& -1\\
0 &-1 &0& 0
\end{array}\right]
=L(D(\{v_2,v_1\};\{v_6,v_5\}), X)_{\{x_{v_2\circ v_6}=0, \, x_{v_1\circ v_5}=0 \}}.
\end{eqnarray*}
}
\end{Example}

\medskip

In a similar way, we can define the digraph $D[U;V]$ that satisfies that 
\[
L(D, X)[U;V]\approx L(D[U;V], X)_{\{x_{u_i\circ v_i}=-m_{(v_i,u_i)}\, | \, i=1,\ldots,s\}}.
\]

\medskip

In the next results, we calculate the invariant $\gamma$ for the path using the previous results on the minors of a generalized Laplacian matrix.
\begin{Corollary}\label{coropath}
Let $P_n$ be the path with $n$ vertices, then $\gamma_{\mathcal{P}}(P_n)=n-1$.
\end{Corollary}
\begin{proof}
It follows because ${\rm det}(L(P_n,X)(1;n))=1$.
\end{proof}

\begin{Corollary}
Let $G$ be a simple graph with $P_n$ as an induced graph, then $\gamma_{\mathcal{P}}(G)\geq n-1$.
\end{Corollary}
\begin{proof}
It follows directly from Proposition~\ref{BasicProp} and Corollary~\ref{coropath}.
\end{proof}
Moreover, in~\cite{trees} was proved that if $T$ is a tree, then $\gamma_{\mathbb Z}(T)$
is equal to the $2$-matching number of the tree, which is a generalization of an induced path.

A crucial open question discussed in~\cite[Section 4]{lorenzini08} and~\cite[Section 4]{directed} 
states that the critical group is cyclic for almost all simple graphs.
However, after computing the critical ideals of all simple graphs with less or equal to $9$ vertices, 
we conjecture that the only simple graph with $\gamma_{\mathcal{P}}(G)=n-1$ is the path with $n$ vertices.
\begin{Conjecture}
Let $G$ be a simple graph with $n$ vertices and $\mathbb{Z}\subseteq \mathcal{P}$, then $\gamma_{\mathcal{P}}(G)=n-1$ if and only if $G=P_n$.
\end{Conjecture}


\subsection{The critical ideals of the cycle}
In this subsection, we will calculate the critical ideals of the cycle with $n$ vertices.
Let $C_n$ be the cycle with $n$ vertices and let $V(C)=\{v_1,v_2,\ldots, v_n\}$ be its vertex set. 
To simplify the notation, we consider the vertices of $C_n$ as the classes modulo $n$.
That is, for instance, the $(n+1)$-th vertex of $C_n$ is the vertex $v_1$.
Clearly, $I_n(C_n,X)$ is generated by the determinant of $L(C_n,X)$.
In Corollary~\ref{detkpc}, the determinant of the generalized Laplacian of $C_n$ was calculated.
Therefore, the critical ideals of $C_n$, when $1\leq i\leq n-1$, are the only ones that remain to calculate.
In order to simplify the notation, we will write $\mathrm{det}(G,X)$ instead of $\mathrm{det}(L(G, X))$.

First we will prove that almost all the critical ideals of $C_n$ are trivial, except for $i$ equal to $n-1$ and $n$.
Also, we will give a minimal set of generators for $I_{n-1}(C_n,X)$, and after that, we will give a reduced Gr\"obner basis for it.

\begin{Theorem}\label{mingen}
If $C_n$ is the cycle with $n$ vertices, then $I_i(C_n,X)=\langle 1\rangle$ for all $1\leq i\leq n-2$. 
Moreover, 
\[
F_k=\{\mathrm{det}(C_n\setminus \{v_{k+1},v_k,v_{k-1}\},X)+x_k,\mathrm{det}(C_n\setminus \{v_k,v_{k+1}v_{k+2}\},X)+x_{k+1},
\mathrm{det}(C_n\setminus \{v_k,v_{k+1}\},X)+1\} 
\]
for all $1\leq  k\leq  n$ is a minimal set of generators for $I_{n-1}(C_n,X)$.
\end{Theorem}
\begin{proof}
Firstly, we will prove that $I_i(C_n,X)=\langle 1\rangle$ for all  $1\leq i\leq n-2$.
Let $D$ be the digraph obtained from the cycle $C_{n-2}$ with $n-2$ vertices, 
when we delete the arcs $\overset{\longrightarrow}{v_1v_2}$ and $\overset{\longrightarrow}{v_{n-2}v_1}$.
It is not difficult to see that 
\[
L(C_n,X)(\{1,n\};\{n-1,n\})\approx L(D,X)_{\{x_1=0\}}.
\]
By applying Theorem~\ref{det} to the digraph $D$, we have that $|\mathrm{det}(L(C_n,X)(\{1,n\};\{n-1,n\}))|=1$ 
and therefore $I_i(C_n,X)=\langle 1\rangle$ for all $1\leq i\leq n-2$.

\medskip

Now, we turn our attention to the $(n-1)$-th critical ideal of $C_n$.
For all $1\leq i,j\leq n$, let $Q_{i,j}=L(C_n,X)(i;j)$ and 
\[
q_{i,j}=
\begin{cases}
-\mathrm{det}(Q_{i,j}) & \text{ if } {\rm lc} (\mathrm{det}(Q_{i,j}))=-1,\\
\mathrm{det}(Q_{i,j}) & \text{ otherwise.}
\end{cases}
\]
Since $L(C_n,X)$ is symmetric, $Q_{i,j} \approx Q_{j,i}$ and $q_{i,j}=q_{j,i}$.
Clearly, the ideal $I_{n-1}(C_n,X)$ is generated by the $n^2$ minors of size $n-1$ of $L(C_n,X)$.

On the other hand, for all $1\leq i\leq n$ and $0\leq s\leq n-1$, 
let $P_{i}^s$ be the induced path with $s$ vertices of $C_n$ consisting of the vertices $v_i,v_{i+1}, \ldots, v_{i+s-1}$.
For instance, in $C_6$ we have that $P_5^4$ consists of the vertices $v_5,v_6,v_1$, and $v_2$.
Note that, if $s=0$, then $P_{i}^s$ is the empty graph.

\begin{Definition}\label{defp}
For all $1\leq i\leq n$ and $2\leq s\leq n$, let $p_{i,s}={\rm det}(P_i^{s-1},X)$.
For technical convenience we will adopt the convention that $p_{i,1}=1$ 
for all $1\leq i\leq n$ and $p_{i,s}=-p_{i+s,-s}$ for all $s\leq 0$.
\end{Definition}
Note that, by taking $s=0$ we get that $p_{i,-0}=-p_{i,0}$ and therefore $p_{i,0}=0$ for all $1\leq i\leq n$.
Also, taking $s=-1$, we get that $p_{i,-1}=-p_{i-1,1}=-1$. 
Finally, note that ${\rm lt}(p_{i,s})=x_ix_{i+1}\cdots x_{i+s-2}$ for all $s\geq 2$ and ${\rm deg}(p_{i,s})=s-1$ for all $s\geq 1$.
The rest of this article relies heavily in several identities that involved the $q_{i,j}$'s and the $p_{i,s}$'s.
In order to do more understandable the meaning of these identities it will be very helpful to think the polynomial
$q_{i,j}$ as a polynomial associated to the edge $v_iv_j$ (see Figure~\ref{figcycle1}).
The first one of these identities is the following:
\begin{Claim}\label{rel1}
If $1 \leq i\leq j \leq n$, then
\[
q_{i,j}=p_{i+1,j-i}+ p_{j+1,n-j+i}.
\]
\end{Claim}
\begin{proof}
At first, if $j=i$, then clearly $Q_{i,i}=L(C_n,X)(i;i)=L(C_n\setminus v_i,X\setminus x_i)= L(P_{i+1}^{n-1},X)$.
Thus $q_{i,i}=p_{i+1,n}=p_{i+1,0}+p_{i+1,n}$.
Now, if $j=i+1$, then by Lemma~\ref{minors}
\[
Q_{i,i+1}=L(C_n,X)(i;i+1)\approx L(D,X)_{\{x_{v_i\circ v_{i+1}}=-1\}},
\]
where $D$ is the digraph obtained from $C_n$ by deleting the arcs 
$\overset{\longrightarrow}{v_iv_{i-1}}, \overset{\longrightarrow}{v_{i+2}v_{i+1}}$ 
and identifying the vertices $i$ and $i+1$.
Since $D$ has only one spanning directed $1$-factor (a directed cycle) and all the other directed 
$1$-factors of $D$ are $1$-factors of $D\setminus v_i\circ v_{i+1}\cong P_{n-2}$, then
by applying Theorem~\ref{det} to $D$ we get that ${\rm det}(L(D,X))_{\{x_{v_i\circ v_{i+1}}=-1\}}=-p_{i+2,n-1}-1$,
and therefore $q_{i,i+1}=p_{i+1,1}+p_{i+2,n-1}$.
Finally, if $j\neq i, i+1$, then by Lemma~\ref{minors}
\[
L(C_n,X)(i;j)\approx L(D,X)_{\{x_{v_i\circ v_j}=0\}},
\]
where $D$ is the digraph obtained from $C_n$ by deleting the arcs 
$\overset{\longrightarrow}{v_iv_{i+1}}, \overset{\longrightarrow}{v_{j-1}v_j},\overset{\longrightarrow}{v_iv_{i-1}}, \overset{\longrightarrow}{v_{j+1}v_j}$ 
and identifying the vertices $i$ and $j$.
Since $D$ has only two directed cycles containing the vertex $v_i\circ v_j$, then by applying Theorem~\ref{det} we get the result.
\end{proof}

From Claim~\ref{rel1} can be interpreted that $q_{i,j}$ is equal to the sum of the determinants of the generalized 
Laplacian matrices of the two paths obtained when we delete the vertices $v_i$ and $v_j$ to $C_n$.
The next algebraic identity will be key for the rest of the article. 
For instance, it will be useful to find a minimal set of generators for $I_{n-1}(C_n,X)$. 

\begin{Claim}\label{rel2}
If $1\leq i \leq n$, $2\leq j\leq n-1$, and $-j\leq s \leq n-j$, then
\begin{equation}\label{rel2eq}
q_{i,i+j+s}=p_{i+j,s+1}\cdot q_{i,i+j}-p_{i+j+1,s}\cdot q_{i,i+j-1}.
\end{equation}
That is, $\langle  \{q_{i,t}\}_{t=1}^n\rangle= \langle q_{i,i+j}, q_{i,i+j-1} \rangle$.
\end{Claim}
\begin{proof}
Note that by Definition~\ref{defp}, the Equation~\ref{rel2eq} we can be written as 
$q_{i,i+j+s}=-p_{i+j+s+1,-s-1}\cdot q_{i,i+j}+p_{i+j+s+1,-s}\cdot q_{i,i+j-1}$ for all $s\leq -1$.
The result is trivial for $s=0$ and $s=-1$.
We divide the proof in two cases: when $s\geq 0$ and $s\leq -1$.
For both cases we will use induction on $s$.
\begin{figure}[h]
\begin{center}
	\begin{tikzpicture}[line width=1.2pt, scale=1.2]
		\draw[color=blue] (166.8:1.4) arc (166.8:201.6:1.4);
		\draw[color=green] (178.4:1.2) arc (178.4:201.6:1.2);
		\draw[color=green] (100:1.4) arc (100:146.4:1.4);
		\draw[color=blue] (100:1.2) arc (100:134.8:1.2);
		\draw (-90+10:1.3) node (v1) [draw, circle, fill=gray, inner sep=0pt, minimum width=3pt] {};
		\draw (-90+10+11.6:1.3) node (v1_1) [draw, circle, fill=black, inner sep=0pt, minimum width=1pt] {};
		\draw (-90+10+23.2:1.3) node (v1_2) [draw, circle, fill=black, inner sep=0pt, minimum width=1pt] {};
		\draw (-90+10+34.8:1.3) node (v1_3) [draw, circle, fill=black, inner sep=0pt, minimum width=1pt] {};
		
		\draw (-90+56.6:1.3) node (v2) [draw, circle, fill=gray, inner sep=0pt, minimum width=3pt] {};
		\draw (-90+56.6+11.6:1.3) node (v2_1) [draw, circle, fill=black, inner sep=0pt, minimum width=1pt] {};
		\draw (-90+56.6+23.2:1.3) node (v2_2) [draw, circle, fill=black, inner sep=0pt, minimum width=1pt] {};
		\draw (-90+56.6+34.8:1.3) node (v2_3) [draw, circle, fill=black, inner sep=0pt, minimum width=1pt] {};		
		
		\draw (-90+103.2:1.3) node (v3) [draw, circle, fill=gray, inner sep=0pt, minimum width=1pt] {};
		\draw (-90+123.2:1.3) node (v4) [draw, circle, fill=gray, inner sep=0pt, minimum width=1pt] {};
		\draw (-90+123.2+11.6:1.3) node (v4_1) [draw, circle, fill=black, inner sep=0pt, minimum width=1pt] {};
		\draw (-90+123.2+23.2:1.3) node (v4_2) [draw, circle, fill=black, inner sep=0pt, minimum width=1pt] {};
		\draw (-90+123.2+34.8:1.3) node (v4_3) [draw, circle, fill=black, inner sep=0pt, minimum width=1pt] {};		
		
		\draw (90-10:1.3) node (v5) [draw, circle, fill=red, inner sep=0pt, minimum width=3pt] {};
		\draw (90+10:1.3) node (v6) [draw, circle, fill=gray, inner sep=0pt, minimum width=1pt] {};
		\draw (90+10+11.6:1.3) node (v6_1) [draw, circle, fill=black, inner sep=0pt, minimum width=1pt] {};
		\draw (90+10+23.2:1.3) node (v6_2) [draw, circle, fill=black, inner sep=0pt, minimum width=1pt] {};
		\draw (90+10+34.8:1.3) node (v6_3) [draw, circle, fill=black, inner sep=0pt, minimum width=1pt] {};
		\draw (90+10:1.41) node (v6p) [draw, color=green, inner sep=0pt, minimum width=0pt] {};
		\draw (90+10:1.19) node (v6pp) [draw, color=blue, inner sep=0pt, minimum width=0pt] {};
		\draw (90+10+34.8:1.19) node (v6_3pp) [draw, color=blue, inner sep=0pt, minimum width=0pt] {};
		
		\draw (-90-123.2:1.3) node (v7) [draw, circle, fill=green, inner sep=0pt, minimum width=3pt] {};
		\draw (-90-123.2:1.4) node (v7p) [draw, color=green, inner sep=0pt, minimum width=0pt] {};
		\draw (-90-103.2:1.3) node (v8) [draw, circle, fill=blue, inner sep=0pt, minimum width=3pt] {};
		\draw (-90-103.2+11.6:1.3) node (v8_1) [draw, circle, fill=black, inner sep=0pt, minimum width=1pt] {};
		\draw (-90-103.2+23.2:1.3) node (v8_2) [draw, circle, fill=black, inner sep=0pt, minimum width=1pt] {};
		\draw (-90-103.2+34.8:1.3) node (v8_3) [draw, circle, fill=black, inner sep=0pt, minimum width=1pt] {};
		\draw (-90-103.2:1.4) node (v8p) [draw, color=blue, inner sep=0pt, minimum width=0pt] {};
		\draw (-90-103.2+34.8:1.4) node (v8_3p) [draw, color=blue, inner sep=0pt, minimum width=0pt] {};
		\draw (-90-103.2+11.6:1.2) node (v8_1pp) [draw, color=green, inner sep=0pt, minimum width=0pt] {};
		\draw (-90-103.2+34.8:1.2) node (v8_3pp) [draw,color=green, inner sep=0pt, minimum width=0pt] {};
		
		\draw (-90-56.6:1.3) node (v9) [draw, circle, fill=red, inner sep=0pt, minimum width=3pt] {};
		\draw (-90-56.6+11.6:1.3) node (v9_1) [draw, circle, fill=black, inner sep=0pt, minimum width=1pt] {};
		\draw (-90-56.6+23.2:1.3) node (v9_2) [draw, circle, fill=black, inner sep=0pt, minimum width=1pt] {};
		\draw (-90-56.6+34.8:1.3) node (v9_3) [draw, circle, fill=black, inner sep=0pt, minimum width=1pt] {};
		
		\draw (-90-10:1.3) node (v10) [draw, circle, fill=gray, inner sep=0pt, minimum width=3pt] {};
		\draw (v1)+(0.1,-0.25) node () {\small $v_1$};
		\draw (v2)+(0.2,-0.2) node () {\small $v_i$};
		\draw (v5)+(0.85,0.2) node () {\small $v_{i+j+s}\, (s\leq -1)$};
		\draw (v7)+(-0.5,-0.05) node () {\small $v_{i+j-1}$};
		\draw[color=green] (v7)+(1.1,-0.3) node () {\small $q_{i,i+j-1}$};
		\draw (v8)+(-0.38,0.08) node () {\small $v_{i+j}$};
		\draw[color=blue] (v8)+(1.2,-0.75) node () {\small $q_{i,i+j}$};
		\draw[color=blue] (v8)+(-0.65,-0.5) node () {\small $p_{i+j,s+1}$};
		\draw[color=green] (v8)+(0.57,-0.50) node () {\tiny $p_{i+j+1,s}$};
		\draw (v9)+(-0.45,-0.15) node () {\small $v_{i+j+s}$};
		\draw (v10)+(0,-0.25) node () {\small $v_{n}$};
		\draw[color=red] (0,-0.9) node () {\small $q_{i,i+j+s}$};
		\draw[color=green] (-1.5,1.3) node () {\small $p_{i+j+s+1,-s}$};
		\draw[color=blue] (0.1,0.9) node () {\tiny $p_{i+j+s+1,-s-1}$};
		\draw (1.45,0.5) node () {\small $C_n$};
		
		\draw (v1) edge[color=gray] (v1_1);
		\draw (v1_1) edge[color=gray] (v1_2);
		\draw (v1_2) edge[color=gray] (v1_3);
		\draw (v1_3) edge[color=gray] (v2);
		
		\draw (v2) edge[color=gray] (v2_1);
		\draw (v2_1) edge[color=gray] (v2_2);
		\draw (v2_2) edge[color=gray] (v2_3);
		\draw (v2_3) edge[color=gray] (v3);
		
		\draw (v4) edge[color=gray] (v4_1);
		\draw (v4_1) edge[color=gray] (v4_2);
		\draw (v4_2) edge[color=gray] (v4_3);
		\draw (v4_3) edge[color=gray] (v5);
		
		\draw (v6) edge[color=gray] (v6_1);
		\draw (v6_1) edge[color=gray] (v6_2);
		\draw (v6_2) edge[color=gray] (v6_3);
		\draw (v6_3) edge[color=gray] (v7);
		
		\draw (v8) edge[color=gray] (v8_1);
		\draw (v8_1) edge[color=gray] (v8_2);
		\draw (v8_2) edge[color=gray] (v8_3);
		\draw (v8_3) edge[color=gray] (v9);
		
		\draw (v9) edge[color=gray] (v9_1);
		\draw (v9_1) edge[color=gray] (v9_2);
		\draw (v9_2) edge[color=gray] (v9_3);
		\draw (v9_3) edge[color=gray] (v10);
		
		\draw (v1) edge (v10);
		\draw (v2) edge[color=red, dashed] (v5);
		\draw (v2) edge[color=green] (v7);
		\draw (v2) edge[color=blue] (v8);
		\draw (v3) edge[color=gray] (v4);   
		\draw (v5) edge[color=gray] (v6);
		\draw (v7) edge (v8);  
		\draw (v2) edge[color=red, dashed] (v9);
		\draw (v8) edge[color=blue] (v8p);
		\draw (v8_3) edge[color=blue] (v8_3p);
		\draw (v8_1) edge[color=green] (v8_1pp);
		\draw (v8_3) edge[color=green] (v8_3pp);
		\draw (v6) edge[color=blue] (v6pp);
		\draw (v6) edge[color=green] (v6p);
		\draw (v6_3) edge[color=blue] (v6_3pp);
		\draw (v7) edge[color=green] (v7p);
	\end{tikzpicture}	
\end{center}
\caption{ A cycle $C_n$ with $n$ vertices illustrating the identity 
$\textcolor{red}{q_{i,i+j+s}}=-\textcolor{blue}{p_{i+j+s+1,-s-1}\cdot q_{i,i+j}}+\textcolor{green}{p_{i+j+s+1,-s}\cdot q_{i,i+j-1}}$, 
on the $q_{i,j}$'s  and $p_{i,s}$'s polynomials given in claim~\ref{rel2}.}
\label{figcycle1}
\end{figure}
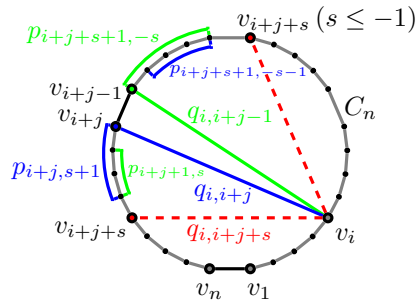

We begin with the case when $s\geq 0$.
If $s=1$, then expanding the determinant of $L(P_{i+1}^{j}, X)$ and $L(P_{i+j}^{n-j}, X)$ we have that
$p_{i+1,j+1}=x_{i+j}\cdot p_{i+1,j}- p_{i+1,j-1}$ and $p_{i+j,n-j+1}=x_{i+j}\cdot p_{i+j+1,n-j}- p_{i+j+2,n-j-1}$.
Therefore
\begin{eqnarray*}
x_{i+j}q_{i,i+j} &\overset{\ref{rel1}}{=}& x_{i+j}\cdot [p_{i+1,j}+ p_{i+j+1,n-j}]
= p_{i+1,j-1}+ p_{i+1,j+1}+p_{i+j,n-j+1}+p_{i+j+2,n-j-1}\\
&=& p_{i+1,j-1}+p_{i+j,n-j+1}+p_{i+1,j+1}+p_{i+j+2,n-j-1}=q_{i,i+j-1}+q_{i,i+j+1}.
\end{eqnarray*}
Now, assume that $s\geq 2$. 
Using induction hypothesis we get
\begin{eqnarray*}
q_{i,i+j+s+1} &=& x_{i+j+s}\cdot q_{i,i+j+s}-q_{i,i+j+s-1}\\
&=& x_{i+j+s}\cdot (p_{i+j,s+1}\cdot q_{i,i+j}-p_{i+j+1,s}\cdot q_{i,i+j-1})- (p_{i+j,s}\cdot q_{i,i+j}-p_{i+j+1,s-1}\cdot q_{i,i+j-1})\\
&=& (x_{i+j+s}\cdot p_{i+j,s+1}-p_{i+j,s})\cdot q_{i,i+j}- (x_{i+j+s}\cdot p_{i+j+1,s}-p_{i+j+1,s-1}) \cdot q_{i,i+j-1}\\
&=& p_{i+j,s+2}\cdot q_{i,i+j}-p_{i+j+1,s+1}\cdot q_{i,i+j-1}.
\end{eqnarray*}

Similar arguments can be used when $s\leq -1$.
\end{proof}

Note that Claim~\ref{rel2} gives us two different expression for $q_{i,i}$.
Namely, $q_{i,i}=p_{i+j,n-j+1}\cdot q_{i,i+j}-p_{i+j+1,n-j}\cdot q_{i,i+j-1}=p_{i+1,j}\cdot q_{i,i+j-1}-p_{i+1,j-1}\cdot q_{i.i+j}$.

Thinking the $q_{i,j}$'s polynomials as a polynomial associated to the edge $v_iv_j$, we have that
Claim~\ref{rel2} says us that the polynomial associated to any edge adjacent to $v_i$
can be expressed as a polynomial combination of the two polynomial associated to two consecutive edges adjacent to $v_i$.
Thus, if we fix $i$ and $j$ (with $i\neq j,j+1$), then for all $1\leq t\leq n$, $q_{i,t}$ can be generated by $q_{i,j}$ and $q_{i,j+1}$.
That is, $\langle  \{q_{i,t}\}_{t=1}^n\rangle= \langle q_{i,j}, q_{i,j+1} \rangle$ for all $1\leq i,j\leq n$ with $i\neq j,j+1$.
Therefore, it is not difficult to see that
\[
I_{n-1}(C_n,X)=\langle q_{i,i+1}, q_{i,i+2},q_{i-1,i+1} \rangle \text{ for all } 1\leq i\leq n. \vspace{-7mm}
\]
\end{proof}

\medskip

Next example presents how claim~\ref{rel2} can be used.
\begin{Example}
For the cycle with $6$ vertices we have that
\begin{multicols}{2}
\begin{enumerate}[\it (i)]
\item $q_{1,1}= x_2 x_3 x_4 x_5 x_6 - x_2 x_3 x_4 - x_2 x_3 x_6 - x_2 x_5 x_6 - x_4 x_5 x_6+x_2 + x_4+ x_6$,
\item $q_{1,2}= x_3 x_4 x_5 x_6 - x_3 x_4 - x_3 x_6 - x_5 x_6+2 $,
\item $q_{1,3}= x_4 x_5 x_6+x_2 - x_4 - x_6 $,
\item $q_{1,4}=  x_2 x_3 +x_5 x_6-2$,
\item $q_{1,5}= x_2 x_3 x_4 -x_2 - x_4+ x_6$,
\item $q_{1,6}= x_2 x_3 x_4 x_5- x_2 x_3 - x_2 x_5 - x_4 x_5 +2$.
\end{enumerate}
\end{multicols}
It is not difficult to see that
\begin{itemize}
\item $q_{1,1}=p_{4,-2}\cdot q_{1,4}-p_{5,-3}\cdot q_{1,3}=-p_{2,2}\cdot q_{1,4}+p_{2,3}\cdot q_{1,3}=-x_2\cdot q_{1,4}+(x_2x_3-1)\cdot q_{1,3}$,
\item $q_{1,2}=p_{4,-1}\cdot q_{1,4}-p_{5,-2}\cdot q_{1,3}=-p_{3,1}\cdot q_{1,4}+p_{3,2}\cdot q_{1,3}=-q_{1,4}+x_3\cdot q_{1,3}$,
\item $q_{1,3}=p_{4,0}\cdot q_{1,4}-p_{5,-1}\cdot q_{1,3}$,
\item $q_{1,4}=p_{4,1}\cdot q_{1,4}-p_{5,0}\cdot q_{1,3}$,
\item $q_{1,5}= p_{4,2}\cdot q_{1,4}-p_{5,1}\cdot q_{1,3}= x_4\cdot q_{1,4}-q_{1,3}$,
\item $q_{1,6}=p_{4,3}\cdot q_{1,4}-p_{5,2}\cdot q_{1,3}=(x_4x_5-1)\cdot q_{1,4}-x_5\cdot q_{1,3}$,
\item $q_{1,1}=p_{4,4}\cdot q_{1,4}-p_{5,3}\cdot q_{1,3}=(x_4x_5x_6-x_4-x_6)\cdot q_{1,4}-(x_5x_6-1)\cdot q_{1,3}$.
\end{itemize}
\end{Example}

\medskip

Before to present a Gr\"obner basis for $I_{n-1}(C_n,X)$, we present an identity between the $q_{i,j}$'s.
These identities will be useful to reduce the $S$-polynomials of the $q_{i,j}$'s.
\begin{Proposition}\label{rel4}
Let $n\geq 6$, $0\leq i,s\leq n$, and $2\leq j,t\leq n-2$. 
\begin{description}
\item[(i)] If $0\leq s\leq j-1$, then (see Figure~\ref{figcycle2} $(i)$)
\[
p_{i+s+1,j-s}\cdot q_{i+j-1, i}-p_{i+s+t,n-s-t+1}\cdot q_{i+s,i+s+t}=p_{i+s+1,j-s-1}\cdot q_{i+j,i}-p_{i+s+t+1,n-s-t}\cdot q_{i+s,i+s+t-1}.
\]
\item[(ii)] If $j\leq s\leq n$, then (see Figure~\ref{figcycle2} $(ii)$)
\[
p_{i+j,s-j+1}\cdot q_{i,i+j}-p_{i+1,s+t-n}\cdot q_{i+s+t-1,i+s}=p_{i+j+1,s-j}\cdot q_{i,i+j-1}-p_{i+1,s+t-n-1}\cdot q_{i+s+t,i+s}.
\]
\end{description}
\end{Proposition}
\begin{proof}
$(i)$ By using Claim~\ref{rel2} when $0\leq s\leq j-1$, we have that
\[
q_{i,i+s}=p_{i+s+1,j-s}\cdot q_{i,i+j-1}-p_{i+s+1,j-s-1}\cdot q_{i,i+j} 
\]
and 
\[
q_{i+s,i}=p_{i+s+t,n-s-t+1}\cdot q_{i+s,i+s+t}-p_{i+s+t+1,n-s-t}\cdot q_{i+s,i+s+t-1}.
\]
By using the fact that $q_{i,j}=q_{j,i}$ in the first identity, we get that
\[
p_{i+s+1,j-s}\cdot q_{i+j-1, i}-p_{i+s+t,n-s-t+1}\cdot q_{i+s,i+s+t}=p_{i+s+1,j-s-1}\cdot q_{i+j,i}-p_{i+s+t+1,n-s-t}\cdot q_{i+s,i+s+t-1}.
\]
\begin{figure}[h]
\begin{center}
\begin{tabular}{c@{\extracolsep{10mm}}c}
	\begin{tikzpicture}[line width=1.2pt, scale=1.2]
		\draw[color=blue] (45:1.2) arc (45:126.8-11.6:1.2);
		\draw[color=green] (45:1.4) arc (45:126.8:1.4);
		\draw[color=red] (-90-56.6:1.4) arc (-90-56.6:-80+34.8:1.4);
		\draw[color=yellow] (-90-56.6+11.6:1.2) arc (-90-56.6+11.6:-80+34.8:1.2);  
		\draw (-80:1.3) node (v1) [draw, circle, fill=gray, inner sep=0pt, minimum width=3pt] {};
		\draw (-80+11.6:1.3) node (v1_1) [draw, circle, fill=black, inner sep=0pt, minimum width=1pt] {};
		\draw (-80+23.2:1.3) node (v1_2) [draw, circle, fill=black, inner sep=0pt, minimum width=1pt] {};
		\draw (-80+34.8:1.3) node (v1_3) [draw, circle, fill=black, inner sep=0pt, minimum width=1pt] {};
		\draw (-80+34.8:1.4) node (v1_3p) [draw, color=red, inner sep=0pt, minimum width=0pt] {};
		\draw (-80+34.8:1.2) node (v1_3pp) [draw, color=yellow, inner sep=0pt, minimum width=0pt] {};
		
		\draw (-90+56.6:1.3) node (v2) [draw, circle, fill=gray, inner sep=0pt, minimum width=3pt] {};
		\draw (-90+56.6+11.6:1.3) node (v2_1) [draw, circle, fill=black, inner sep=0pt, minimum width=1pt] {};
		\draw (-90+56.6+23.2:1.3) node (v2_2) [draw, circle, fill=black, inner sep=0pt, minimum width=1pt] {};
		\draw (-90+56.6+34.8:1.3) node (v2_3) [draw, circle, fill=black, inner sep=0pt, minimum width=1pt] {};	
		
		\draw (-90+103.2:1.3) node (v3) [draw, circle, fill=black, inner sep=0pt, minimum width=0pt] {};
		\draw (-90+103.2+11.6:1.3) node (v3_1) [draw, circle, fill=black, inner sep=0pt, minimum width=1pt] {};
		\draw (90-56.6:1.3) node (v3_2) [draw, circle, fill=red, inner sep=0pt, minimum width=3pt] {};
		\draw (90-56.6+11.6:1.3) node (v3_3) [draw, circle, fill=black, inner sep=0pt, minimum width=1pt] {};
		\draw (90-56.6+11.6:1.4) node (v3_3p) [draw, color=green, inner sep=0pt, minimum width=0pt] {};
		\draw (90-56.6+11.6:1.2) node (v3_3pp) [draw, color=blue, inner sep=0pt, minimum width=0pt] {};
		\draw (90-56.6+23.2:1.3) node (v3_4) [draw, circle, fill=black, inner sep=0pt, minimum width=1pt] {};		
		
		\draw (90-21.4:1.3) node (v4) [draw, circle, fill=gray, inner sep=0pt, minimum width=1pt] {};
		\draw (90-10:1.3) node (v5) [draw, circle, fill=gray, inner sep=0pt, minimum width=1pt] {};
		\draw (90-10+11.6:1.3) node (v5_1) [draw, circle, fill=black, inner sep=0pt, minimum width=1pt] {};
		\draw (90-10+23.2:1.3) node (v5_2) [draw, circle, fill=black, inner sep=0pt, minimum width=1pt] {};
		\draw (90-10+34.8:1.3) node (v5_3) [draw, circle, fill=black, inner sep=0pt, minimum width=1pt] {};
		\draw (90-10+34.8:1.2) node (v5_3p) [draw, color=blue, inner sep=0pt, minimum width=0pt] {};
		
		\draw (-90-143.2:1.3) node (v6) [draw, circle, fill=green, inner sep=0pt, minimum width=3pt] {};
		\draw (-90-143.2:1.4) node (v6p) [draw, color=green, inner sep=0pt, minimum width=0pt] {};
		\draw (-90-123.2:1.3) node (v7) [draw, circle, fill=blue, inner sep=0pt, minimum width=3pt] {};
		\draw (-90-123.2+11.6:1.3) node (v7_1) [draw, circle, fill=black, inner sep=0pt, minimum width=1pt] {};
		\draw (-90-123.2+23.2:1.3) node (v7_2) [draw, circle, fill=black, inner sep=0pt, minimum width=1pt] {};
		\draw (-90-123.2+34.8:1.3) node (v7_3) [draw, circle, fill=black, inner sep=0pt, minimum width=1pt] {};
		\draw (-90-76.6:1.3) node (v8) [draw, circle, fill=gray, inner sep=0pt, minimum width=3pt] {};
		
		\draw (-90-56.6:1.3) node (v9) [draw, circle, fill=gray, inner sep=0pt, minimum width=3pt] {};
		\draw (-90-56.6:1.4) node (v9p) [draw, color=red, inner sep=0pt, minimum width=0pt] {};
		\draw (-90-56.6+11.6:1.3) node (v9_1) [draw, circle, fill=black, inner sep=0pt, minimum width=1pt] {};
		\draw (-90-56.6+11.6:1.2) node (v9_1pp) [draw, color=yellow, inner sep=0pt, minimum width=0pt] {};
		\draw (-90-56.6+23.2:1.3) node (v9_2) [draw, circle, fill=black, inner sep=0pt, minimum width=1pt] {};
		\draw (-90-56.6+34.8:1.3) node (v9_3) [draw, circle, fill=black, inner sep=0pt, minimum width=1pt] {};
		
		\draw (-90-10:1.3) node (v10) [draw, circle, fill=gray, inner sep=0pt, minimum width=3pt] {};
		\draw (v1)+(0.1,-0.25) node () {\small $v_1$};
		\draw (v10)+(0,-0.25) node () {\small $v_{n}$};
		\draw (v2)+(0.2,-0.2) node () {\small $v_i$};
		\draw (v3_2)+(0.35,0) node () {\small $v_{i+s}$};
		\draw (v9)+(-0.55,0) node () {\small $v_{i+s+t}$};
		\draw (v8)+(-0.6,0) node () {\small $v_{i+s+t-1}$};
		\draw (v6)+(-0.5,0.1) node () {\small $v_{i+j-1}$};
		\draw (v7)+(-0.4,0) node () {\small $v_{i+j}$};
		\draw[color=green] (0,1.55) node () {\small $p_{i+s+1,j-s}$};
		\draw[color=blue] (0.3,0.8) node () {\tiny $p_{i+s+1,j-s-1}$};
		\draw[color=red] (-1.7,-1.3) node () {\small $p_{i+s+t,n-s-t+1}$};
		\draw[color=yellow] (0,-0.75) node () {\tiny $p_{i+s+t+1,n-s-t}$};
		\draw (0.75,0.1) node () {\small $q_{i,i+s}$};
		
		\draw (v1) edge[color=gray] (v1_1);
		\draw (v1_1) edge[color=gray] (v1_2);
		\draw (v1_2) edge[color=gray] (v1_3);
		\draw (v1_3) edge[color=gray] (v2);
		
		\draw (v2) edge[color=gray] (v2_1);
		\draw (v2_1) edge[color=gray] (v2_2);
		\draw (v2_2) edge[color=gray] (v2_3);
		\draw (v2_3) edge[color=gray] (v3);
		
		\draw (v3) edge[color=gray] (v3_1);
		\draw (v3_1) edge[color=gray] (v3_2);
		\draw (v3_2) edge[color=gray] (v3_3);
		\draw (v3_3) edge[color=gray] (v3_4);
		\draw (v3_4) edge[color=gray] (v4);
		
		\draw (v4) edge[color=gray] (v5);
		\draw (v5) edge[color=gray] (v5_1);
		\draw (v5_1) edge[color=gray] (v5_2);
		\draw (v5_2) edge[color=gray] (v5_3);
		\draw (v5_3) edge[color=gray] (v6);
		
		\draw (v6) edge[color=gray] (v7);
		\draw (v7) edge[color=gray] (v7_1);
		\draw (v7_1) edge[color=gray] (v7_2);
		\draw (v7_2) edge[color=gray] (v7_3);
		\draw (v7_3) edge[color=gray] (v8);
		
		\draw (v8) edge[color=gray] (v9);
		\draw (v9) edge[color=gray] (v9_1);
		\draw (v9_1) edge[color=gray] (v9_2);
		\draw (v9_2) edge[color=gray] (v9_3);
		\draw (v9_3) edge[color=gray] (v10);
		\draw (v1) edge (v10);
		
		\draw (v2) edge[color=green] (v6);
		\draw (v2) edge[color=blue] (v7);
		\draw (v2) edge[color=gray, dashed] (v3_2);
		\draw (v3_2) edge[color=red] (v9);
		\draw (v3_2) edge[color=yellow] (v8);
		\draw (v9_1) edge[color=yellow] (v9_1pp);
		\draw (v6) edge[color=green] (v6p);
		\draw (v3_3) edge[color=green] (v3_3p);
		\draw (v1_3) edge[color=yellow] (v1_3pp); 
		\draw (v3_3) edge[color=blue] (v3_3pp);
		\draw (v5_3) edge[color=blue] (v5_3p);
		\draw (v9) edge[color=red] (v9p);
		\draw (v1_3) edge[color=red] (v1_3p); 	 
	\end{tikzpicture}	
&
	\begin{tikzpicture}[line width=1.2pt, scale=1.2]
		\draw[color=blue] (166.8:1.4) arc (166.8:201.6:1.4);
		\draw[color=green] (178.4:1.2) arc (178.4:201.6:1.2);
		\draw[color=red] (-90+56.6+11.6:1.4) arc (-90+56.6+11.6:-90+103.2:1.4);
		\draw[color=yellow] (-90+56.6+11.6:1.2) arc (-90+56.6+11.6:-90+56.6+34.8:1.2);  
		\draw (-90+10:1.3) node (v1) [draw, circle, fill=gray, inner sep=0pt, minimum width=3pt] {};
		\draw (-90+10+11.6:1.3) node (v1_1) [draw, circle, fill=black, inner sep=0pt, minimum width=1pt] {};
		\draw (-90+10+23.2:1.3) node (v1_2) [draw, circle, fill=black, inner sep=0pt, minimum width=1pt] {};
		\draw (-90+10+34.8:1.3) node (v1_3) [draw, circle, fill=black, inner sep=0pt, minimum width=1pt] {};
		
		\draw (-90+56.6:1.3) node (v2) [draw, circle, fill=gray, inner sep=0pt, minimum width=3pt] {};
		\draw (-90+56.6+11.6:1.3) node (v2_1) [draw, circle, fill=black, inner sep=0pt, minimum width=1pt] {};
		\draw (-90+56.6+11.6:1.4) node (v2_1p) [draw, color=red, inner sep=0pt, minimum width=0pt] {};
		\draw (-90+56.6+11.6:1.2) node (v2_1pp) [draw, color=yellow, inner sep=0pt, minimum width=0pt] {};
		\draw (-90+56.6+23.2:1.3) node (v2_2) [draw, circle, fill=black, inner sep=0pt, minimum width=1pt] {};
		\draw (-90+56.6+34.8:1.3) node (v2_3) [draw, circle, fill=black, inner sep=0pt, minimum width=1pt] {};
		\draw (-90+56.6+34.8:1.2) node (v2_3pp) [draw, color=yellow, inner sep=0pt, minimum width=0pt] {};		
		
		\draw (-90+103.2:1.3) node (v3) [draw, circle, fill=red, inner sep=0pt, minimum width=3pt] {};
		\draw (-90+103.2:1.4) node (v3p) [draw, color=red, inner sep=0pt, minimum width=0pt] {};
		\draw (-90+123.2:1.3) node (v4) [draw, circle, fill=yellow, inner sep=0pt, minimum width=3pt] {};
		\draw (-90+123.2+11.6:1.3) node (v4_1) [draw, circle, fill=black, inner sep=0pt, minimum width=1pt] {};
		\draw (-90+123.2+23.2:1.3) node (v4_2) [draw, circle, fill=black, inner sep=0pt, minimum width=1pt] {};
		\draw (-90+123.2+34.8:1.3) node (v4_3) [draw, circle, fill=black, inner sep=0pt, minimum width=1pt] {};		
		
		\draw (90-10:1.3) node (v5) [draw, circle, fill=gray, inner sep=0pt, minimum width=1pt] {};
		\draw (90+10:1.3) node (v6) [draw, circle, fill=gray, inner sep=0pt, minimum width=1pt] {};
		\draw (90+10+11.6:1.3) node (v6_1) [draw, circle, fill=black, inner sep=0pt, minimum width=1pt] {};
		\draw (90+10+23.2:1.3) node (v6_2) [draw, circle, fill=black, inner sep=0pt, minimum width=1pt] {};
		\draw (90+10+34.8:1.3) node (v6_3) [draw, circle, fill=black, inner sep=0pt, minimum width=1pt] {};
		
		\draw (-90-123.2:1.3) node (v7) [draw, circle, fill=green, inner sep=0pt, minimum width=3pt] {};
		\draw (-90-103.2:1.3) node (v8) [draw, circle, fill=blue, inner sep=0pt, minimum width=3pt] {};
		\draw (-90-103.2+11.6:1.3) node (v8_1) [draw, circle, fill=black, inner sep=0pt, minimum width=1pt] {};
		\draw (-90-103.2+23.2:1.3) node (v8_2) [draw, circle, fill=black, inner sep=0pt, minimum width=1pt] {};
		\draw (-90-103.2+34.8:1.3) node (v8_3) [draw, circle, fill=black, inner sep=0pt, minimum width=1pt] {};
		\draw (-90-103.2:1.4) node (v8p) [draw, color=blue, inner sep=0pt, minimum width=0pt] {};
		\draw (-90-103.2+34.8:1.4) node (v8_3p) [draw, color=blue, inner sep=0pt, minimum width=0pt] {};
		\draw (-90-103.2+11.6:1.2) node (v8_1pp) [draw, color=green, inner sep=0pt, minimum width=0pt] {};
		\draw (-90-103.2+34.8:1.2) node (v8_3pp) [draw,color=green, inner sep=0pt, minimum width=0pt] {};
		
		\draw (-90-56.6:1.3) node (v9) [draw, circle, fill=gray, inner sep=0pt, minimum width=3pt] {};
		\draw (-90-56.6+11.6:1.3) node (v9_1) [draw, circle, fill=black, inner sep=0pt, minimum width=1pt] {};
		\draw (-90-56.6+23.2:1.3) node (v9_2) [draw, circle, fill=black, inner sep=0pt, minimum width=1pt] {};
		\draw (-90-56.6+34.8:1.3) node (v9_3) [draw, circle, fill=black, inner sep=0pt, minimum width=1pt] {};
		
		\draw (-90-10:1.3) node (v10) [draw, circle, fill=gray, inner sep=0pt, minimum width=3pt] {};
		\draw (v1)+(0.1,-0.25) node () {\small $v_1$};
		\draw (v2)+(0.2,-0.2) node () {\small $v_i$};
		\draw (v3)+(0.67,0) node () {\small $v_{i+s+t-1}$};
		\draw (v4)+(0.49,0) node () {\small $v_{i+s+t}$};
		\draw (v7)+(-0.5,0.1) node () {\small $v_{i+j-1}$};
		\draw[color=green] (v7)+(1.1,-0.3) node () {\small $q_{i,i+j-1}$};
		\draw[color=yellow] (v7)+(1.4,0) node () {\small $q_{i+s,i+s+t}$};
		\draw (v8)+(-0.4,0) node () {\small $v_{i+j}$};
		\draw[color=blue] (v8)+(-0.85,-0.5) node () {\small $p_{i+j,s-j+1}$};
		\draw[color=red] (v8)+(3.4,-0.5) node () {\small $p_{i+1,s+t-n}$};
		\draw[color=blue] (v8)+(0.5,-0.47) node () {\small $q_{i,i+j}$};
		\draw (v9)+(-0.25,-0.15) node () {\small $v_{i+s}$};
		\draw (v10)+(0,-0.25) node () {\small $v_{n}$};
		\draw (0,-0.9) node () {\small $q_{i,i+s}$};
		
		\draw (v1) edge[color=gray] (v1_1);
		\draw (v1_1) edge[color=gray] (v1_2);
		\draw (v1_2) edge[color=gray] (v1_3);
		\draw (v1_3) edge[color=gray] (v2);
		
		\draw (v2) edge[color=gray] (v2_1);
		\draw (v2_1) edge[color=gray] (v2_2);
		\draw (v2_2) edge[color=gray] (v2_3);
		\draw (v2_3) edge[color=gray] (v3);
		
		\draw (v4) edge[color=gray] (v4_1);
		\draw (v4_1) edge[color=gray] (v4_2);
		\draw (v4_2) edge[color=gray] (v4_3);
		\draw (v4_3) edge[color=gray] (v5);
		
		\draw (v6) edge[color=gray] (v6_1);
		\draw (v6_1) edge[color=gray] (v6_2);
		\draw (v6_2) edge[color=gray] (v6_3);
		\draw (v6_3) edge[color=gray] (v7);
		
		\draw (v8) edge[color=gray] (v8_1);
		\draw (v8_1) edge[color=gray] (v8_2);
		\draw (v8_2) edge[color=gray] (v8_3);
		\draw (v8_3) edge[color=gray] (v9);
		
		\draw (v9) edge[color=gray] (v9_1);
		\draw (v9_1) edge[color=gray] (v9_2);
		\draw (v9_2) edge[color=gray] (v9_3);
		\draw (v9_3) edge[color=gray] (v10);
		
		\draw (v1) edge (v10);
		\draw (v2) edge[color=green] (v7);
		\draw (v2) edge[color=blue] (v8);
		\draw (v3) edge (v4);   
		\draw (v5) edge[color=gray] (v6);
		\draw (v7) edge (v8);  
		\draw (v2) edge[color=gray, dashed] (v9);
		\draw (v3) edge[color=red] (v9);
		\draw (v4) edge[color=yellow] (v9);
		\draw (v3) edge[color=red] (v3p);
		\draw (v2_1) edge[color=red] (v2_1p);
		\draw (v2_3) edge[color=yellow] (v2_3pp);
		\draw (v2_1) edge[color=yellow] (v2_1pp); 
		\draw (v8) edge[color=blue] (v8p);
		\draw (v8_3) edge[color=blue] (v8_3p);
		\draw (v8_1) edge[color=green] (v8_1pp);
		\draw (v8_3) edge[color=green] (v8_3pp); 	 
	\end{tikzpicture}
\\
$(i)$ $0\leq s\leq j-1$,
&
$(ii)$ $j\leq s\leq n$.		
\end{tabular}	
\end{center}
\caption{ A cycle with $n$ vertices illustrating the identity 
given in Proposition~\ref{rel4}.}
\label{figcycle2}
\end{figure}
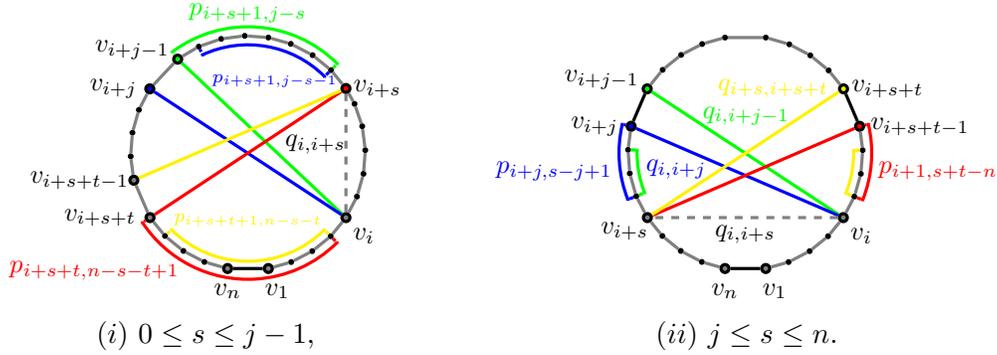

$(ii)$ In a similar way, if $j\leq s\leq n$, then by Claim~\ref{rel2}
\[
q_{i,i+s}=p_{i+j,s-j+1}\cdot q_{i,i+j}-p_{i+j+1,s-j}\cdot q_{i,i+j-1} 
\]
and 
\[
q_{i+s,i}=p_{i+1,s+t-n}\cdot q_{i+s,i+s+t-1}-p_{i+1,s+t-n-1}\cdot q_{i+s,i+s+t}
\]
By using the fact that $q_{i,j}=q_{j,i}$ in the second identity we get that
\[
p_{i+j,s-j+1}\cdot q_{i,i+j}-p_{i+1,s+t-n}\cdot q_{i+s+t-1,i+s}=p_{i+j+1,s-j}\cdot q_{i,i+j-1}-p_{i+1,s+t-n-1}\cdot q_{i+s+t,i+s}.
\]
\end{proof}

\begin{Remark}
Note that the two identities in Proposition~\ref{rel4} are equivalents in the sense that only differ by a rotation on the labels of the vertices.
\end{Remark}

Now, we are ready to find a Gr\"obner basis for $I_{n-1}(C_n,X)$.
In general, Gr\"obner basis is very useful, see for instance~\cite{cox}.
For instance, almost all the information about an ideal can be extracted from its Gr\"obner basis.
We divide this into odd and even cases.

\begin{Theorem}\label{Grobnerodd}
Let $n=2m+1\geq 7$ and
\[
B_1=\{q_{i,i+m+1}\, | \, i=1,\ldots, n\} =\{q_{i,i+m}\, | \, i=1,\ldots, n\}.
\]
Then $B_1$ is a reduced Gr\"obner basis for $I_{n-1}(C_n,X)$ with respect to any graded lexicographic order.
\end{Theorem}
\begin{proof}
First, since $q_{i,i+m}$ and $q_{i,i+m+1}$ are in $B_1$ for all $1\leq i\leq n$, then it is not difficult to see that $I_{n-1}(C_n,X)=\langle B_1\rangle$.
On the other hand, by Claim~\ref{rel1}, ${\rm lt}(q_{i,i+m+1})={\rm lt}(p_{i+1,m+1})=x_{i+1}\cdots x_{i+m}$.
Thus ${\rm deg}(q_{i,i+m+1})=m$ and ${\rm lt}(q_{i,i+m+1}) \nmid \,{\rm lt}(q_{i',i'+m+1})$ for all $1\leq i<i'\leq n$.
Moreover, ${\rm deg}(q_{i,i+m+1}-{\rm lt}(q_{i,i+m+1}) )=m-1$ for all $1\leq i \leq n$.
Therefore, if $B_1$ is a Gr\"obner basis for $I_{n-1}(C_n,X)$, then it is reduced.

In order to prove that $B_1$ is a Gr\"obner basis for $I_{n-1}(C_n,X)$, we need to show that all the $S$-polynomials of the elements
on $B_1$ can be reduced to $0$ by elements on $B_1$.
Let $1\leq i\leq n$ and $1\leq r\leq n-1$, it is not difficult to see (Figure~\ref{figcycle3}) that
{\footnotesize 
\begin{eqnarray*}
S(q_{i,i+m+1},q_{i+r,i+m+r+1})&=& 
\begin{cases}
(x_{i+m+1}\cdots x_{i+m+r})\cdot q_{i,i+m+1}-(x_{i+1}\cdots x_{i+r})\cdot q_{i+r,i+m+r+1} &\text{ if } 1\leq r\leq m,\\
(x_{i+r+1}\cdots x_{i+n})\cdot q_{i,i+m+1}-(x_{i+r+m+1}\cdots x_{i+m})\cdot q_{i+r,i+m+r+1} &\text{ if } m+1\le r \leq 2m,\\
\end{cases}
\\
&=& 
\begin{cases}
{\rm lt}(p_{i+m+1,r+1})\cdot q_{i,i+m+1}-{\rm lt}(p_{i+1,r+1})\cdot q_{i+r,i+m+r+1} &\text{ if } 1\leq r\leq m,\\
{\rm lt}(p_{i+r+1,n-r+1})\cdot q_{i,i+m+1}-{\rm lt}(p_{i+r+m+1,n-r+1})\cdot q_{i+r,i+m+r+1} & \text{ if } m+1\le r \leq 2m.
\end{cases}
\end{eqnarray*}
}
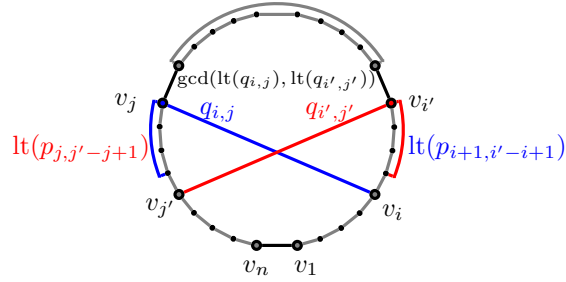
\begin{figure}[h]
\begin{center}
	\begin{tikzpicture}[line width=1.2pt, scale=1.2]
		\draw[color=blue] (166.8:1.4) arc (166.8:201.6:1.4);
		\draw[color=red] (-90+56.6+11.6:1.4) arc (-90+56.6+11.6:-90+103.2:1.4);
		\draw[color=gray] (-90+123.2:1.4) arc (-90+123.2:146.8:1.4);  
		\draw (-90+10:1.3) node (v1) [draw, circle, fill=gray, inner sep=0pt, minimum width=3pt] {};
		\draw (-90+10+11.6:1.3) node (v1_1) [draw, circle, fill=black, inner sep=0pt, minimum width=1pt] {};
		\draw (-90+10+23.2:1.3) node (v1_2) [draw, circle, fill=black, inner sep=0pt, minimum width=1pt] {};
		\draw (-90+10+34.8:1.3) node (v1_3) [draw, circle, fill=black, inner sep=0pt, minimum width=1pt] {};
		
		\draw (-90+56.6:1.3) node (v2) [draw, circle, fill=gray, inner sep=0pt, minimum width=3pt] {};
		\draw (-90+56.6+11.6:1.3) node (v2_1) [draw, circle, fill=black, inner sep=0pt, minimum width=1pt] {};
		\draw (-90+56.6+11.6:1.4) node (v2_1p) [draw, color=red, inner sep=0pt, minimum width=0pt] {};
		\draw (-90+56.6+23.2:1.3) node (v2_2) [draw, circle, fill=black, inner sep=0pt, minimum width=1pt] {};
		\draw (-90+56.6+34.8:1.3) node (v2_3) [draw, circle, fill=black, inner sep=0pt, minimum width=1pt] {};		
		
		\draw (-90+103.2:1.3) node (v3) [draw, circle, fill=red, inner sep=0pt, minimum width=3pt] {};
		\draw (-90+103.2:1.4) node (v3p) [draw, color=red, inner sep=0pt, minimum width=0pt] {};
		\draw (-90+123.2:1.3) node (v4) [draw, circle, fill=gray, inner sep=0pt, minimum width=3pt] {};
		\draw (-90+123.2:1.4) node (v4p) [draw, color=gray, inner sep=0pt, minimum width=0pt] {};
		\draw (-90+123.2+11.6:1.3) node (v4_1) [draw, circle, fill=black, inner sep=0pt, minimum width=1pt] {};
		\draw (-90+123.2+23.2:1.3) node (v4_2) [draw, circle, fill=black, inner sep=0pt, minimum width=1pt] {};
		\draw (-90+123.2+34.8:1.3) node (v4_3) [draw, circle, fill=black, inner sep=0pt, minimum width=1pt] {};		
		
		\draw (90-10:1.3) node (v5) [draw, circle, fill=gray, inner sep=0pt, minimum width=1pt] {};
		\draw (90+10:1.3) node (v6) [draw, circle, fill=gray, inner sep=0pt, minimum width=1pt] {};
		\draw (90+10+11.6:1.3) node (v6_1) [draw, circle, fill=black, inner sep=0pt, minimum width=1pt] {};
		\draw (90+10+23.2:1.3) node (v6_2) [draw, circle, fill=black, inner sep=0pt, minimum width=1pt] {};
		\draw (90+10+34.8:1.3) node (v6_3) [draw, circle, fill=black, inner sep=0pt, minimum width=1pt] {};
		
		\draw (-90-123.2:1.3) node (v7) [draw, circle, fill=gray, inner sep=0pt, minimum width=3pt] {};
		\draw (-90-123.2:1.4) node (v7p) [draw, color=gray, inner sep=0pt, minimum width=0pt] {};
		\draw (-90-103.2:1.3) node (v8) [draw, circle, fill=blue, inner sep=0pt, minimum width=3pt] {};
		\draw (-90-103.2+11.6:1.3) node (v8_1) [draw, circle, fill=black, inner sep=0pt, minimum width=1pt] {};
		\draw (-90-103.2+23.2:1.3) node (v8_2) [draw, circle, fill=black, inner sep=0pt, minimum width=1pt] {};
		\draw (-90-103.2+34.8:1.3) node (v8_3) [draw, circle, fill=black, inner sep=0pt, minimum width=1pt] {};
		\draw (-90-103.2:1.4) node (v8p) [draw, color=blue, inner sep=0pt, minimum width=0pt] {};
		\draw (-90-103.2+34.8:1.4) node (v8_3p) [draw, color=blue, inner sep=0pt, minimum width=0pt] {};
		
		\draw (-90-56.6:1.3) node (v9) [draw, circle, fill=gray, inner sep=0pt, minimum width=3pt] {};
		\draw (-90-56.6+11.6:1.3) node (v9_1) [draw, circle, fill=black, inner sep=0pt, minimum width=1pt] {};
		\draw (-90-56.6+23.2:1.3) node (v9_2) [draw, circle, fill=black, inner sep=0pt, minimum width=1pt] {};
		\draw (-90-56.6+34.8:1.3) node (v9_3) [draw, circle, fill=black, inner sep=0pt, minimum width=1pt] {};
		
		\draw (-90-10:1.3) node (v10) [draw, circle, fill=gray, inner sep=0pt, minimum width=3pt] {};
		\draw (v1)+(0.1,-0.25) node () {\small $v_1$};
		\draw (v2)+(0.2,-0.2) node () {\small $v_i$};
		\draw (v3)+(0.35,0) node () {\small $v_{i'}$};
		\draw (v8)+(-0.4,0) node () {\small $v_{j}$};
		\draw[color=red] (v8)+(-0.9,-0.5) node () {\small ${\rm lt}(p_{j,j'-j+1})$};
		\draw[color=blue] (v8)+(3.6,-0.5) node () {\small ${\rm lt}(p_{i+1,i'-i+1})$};
		\draw[color=blue] (-0.65,0.2) node () {\small $q_{i,j}$};
		\draw[color=red] (0.6,0.2) node () {\small $q_{i',j'}$};
		\draw (0,0.55) node () {\tiny ${\rm gcd}({\rm lt}(q_{i,j}),{\rm lt}(q_{i',j'}))$};
		\draw (v9)+(-0.2,-0.15) node () {\small $v_{j'}$};
		\draw (v10)+(0,-0.25) node () {\small $v_{n}$};
		
		\draw (v1) edge[color=gray] (v1_1);
		\draw (v1_1) edge[color=gray] (v1_2);
		\draw (v1_2) edge[color=gray] (v1_3);
		\draw (v1_3) edge[color=gray] (v2);
		
		\draw (v2) edge[color=gray] (v2_1);
		\draw (v2_1) edge[color=gray] (v2_2);
		\draw (v2_2) edge[color=gray] (v2_3);
		\draw (v2_3) edge[color=gray] (v3);
		
		\draw (v4) edge[color=gray] (v4_1);
		\draw (v4_1) edge[color=gray] (v4_2);
		\draw (v4_2) edge[color=gray] (v4_3);
		\draw (v4_3) edge[color=gray] (v5);
		
		\draw (v6) edge[color=gray] (v6_1);
		\draw (v6_1) edge[color=gray] (v6_2);
		\draw (v6_2) edge[color=gray] (v6_3);
		\draw (v6_3) edge[color=gray] (v7);
		
		\draw (v8) edge[color=gray] (v8_1);
		\draw (v8_1) edge[color=gray] (v8_2);
		\draw (v8_2) edge[color=gray] (v8_3);
		\draw (v8_3) edge[color=gray] (v9);
		
		\draw (v9) edge[color=gray] (v9_1);
		\draw (v9_1) edge[color=gray] (v9_2);
		\draw (v9_2) edge[color=gray] (v9_3);
		\draw (v9_3) edge[color=gray] (v10);
		
		\draw (v1) edge (v10);
		\draw (v2) edge[color=blue] (v8);
		\draw (v3) edge (v4);
		\draw (v4) edge[color=gray] (v4p);   
		\draw (v5) edge[color=gray] (v6);
		\draw (v7) edge (v8); 
		\draw (v7) edge[color=gray] (v7p);   
		\draw (v3) edge[color=red] (v9);
		\draw (v3) edge[color=red] (v3p);
		\draw (v2_1) edge[color=red] (v2_1p); 
		\draw (v8) edge[color=blue] (v8p);
		\draw (v8_3) edge[color=blue] (v8_3p);	 
	\end{tikzpicture}	
\end{center}
\vspace{-3mm}
\caption{The $S$-polynomial of $q_{i,j}$ and $q_{i',j'}$.}
\label{figcycle3}
\end{figure}
By Proposition~\ref{rel4} with $j=t=m+1$ and $s=m+r+1$ in the second identity and $j=t=m+1$, $s=i+m+r+1$ in the first identity we get that
{\footnotesize
\[
S(q_{i,i+m+1},q_{i+r,i+m+r+1})=
\begin{cases}
 -{\rm red}(p_{i+m+1,r+1})\cdot q_{i,i+m+1}+ {\rm red}(p_{i+1,r+1})\cdot q_{i+r,i+m+r+1} +p_{i+m+2,r}\cdot q_{i,i+m}\\
 -p_{i+1,r}\cdot q_{i+r+1,i+m+r+1}, \hspace{66mm} \text{ if } 1\leq r\leq m,\\
 
 -{\rm red}(p_{i+r+1,n-r+1})\cdot q_{i,i+m+1}+{\rm red}(p_{i+r+m+1,n-r+1})\cdot q_{i+r,i+m+r+1}\\ 
 +p_{i+r+1,n-r}\cdot q_{i+1,i+m+1}-p_{i+m+r+2,n-r}\cdot q_{i+r,i+m+r}, \hspace{17mm} \text{ if } m+1\leq r\leq 2m,\\
\end{cases}
\]
}
where ${\rm red}(p)=p-{\rm lt}(p)$ for any polynomial $p$.
 
Since $q_{i,i+m+1}, q_{i+r,i+m+r+1}, q_{i,i+m}, q_{i+r+1,i+m+r+1}, q_{i+1,i+m+1}, q_{i+r,i+m+r}  \in B_1$ for all $1\leq i \leq n$ and $1\leq r\leq n-1$, then
in order to prove that $S(q_{i,i+m+1},q_{i+r,i+m+r+1})\rightarrow_{B_1} 0$, it only remains to prove that the leading terms
of the summands of previous identity are different.
Since ${\rm lt}(p_{i+m+1,r+1}\cdot q_{i,i+m+1})={\rm lt}(p_{i+1,r+1}\cdot q_{i+r,i+m+r+1})$, 
${\rm lt}(p_{i+r+1,n-r+1}\cdot q_{i,i+m+1})={\rm lt}(p_{i+r+m+1,n-r+1}\cdot q_{i+r,i+m+r+1})$
are square free, ${\rm gcd}({\rm lt}(p_{i+m+1,r+1}), {\rm lt}(p_{i+1,r+1}))=1$, 
${\rm gcd}({\rm lt}(p_{i+r+1,n-r+1}), {\rm lt}(p_{i+r+m+1,n-r+1}))=1$,
\[
{\rm red}(p_{i+1,r+1})=x_{i+1}\cdots x_{i+r}/x_{j}x_{j+1} \text{ for some } i+1\leq j\leq i+r-1,
\] 
and
\begin{enumerate}[(i)]
\item ${\rm lt}(p_{i+m+2,r}\cdot q_{i,i+m})=x_{i+m+1}x_{i+m+2}^2\cdots x_{i+m+r}^2x_{i+m+r+1}\cdots x_{i-1}$,
\item ${\rm lt}(p_{i+1,r}\cdot q_{i+r+1,i+m+r+1})=x_{i+m+r+2}\cdots x_{i}x_{i+1}^2\cdots x_{i+r-1}^2x_{i+r}$,
\item ${\rm lt}(p_{i+r+1,n-r}\cdot q_{i+1,i+m+1})=x_{i+m+2}\cdots x_{i+r} x_{i+r+1}^2\cdots x_{i-1}^2x_{i}$,
\item ${\rm lt}(p_{i+m+r+2,n-r}\cdot q_{i+r,i+m+r})=x_{i+m+r+1}x_{i+m+r+2}^2\cdots x_{i+m}^2x_{i+m+1}\cdots x_{i+r-1}$,
\end{enumerate}
then $S(q_{i,i+m+1},q_{i+r,i+r+m+1})\rightarrow_{B_1} 0$ for all $1\leq i\leq  n-1$ and $1\leq r \leq n-i$.
\end{proof}

If $n=3$, then $C_n$ is the complete graph with three vertices.
Now we will present the case of the cycle with five vertices.
\begin{Example}
Let $n=5=2(2)+1=2m+1$ and $B_1$ as in Theorem~\ref{Grobnerodd}.
It is not difficult to compute that
\begin{multicols}{3}
\begin{enumerate}[\it (i)]
\item $q_{1,4}= x_2 x_3-1+x_5$,
\item $q_{2,5}= x_3 x_4-1+x_1$,
\item $q_{3,1}= x_4 x_5-1+x_2$,
\item $q_{4,2}= x_5 x_1-1+x_3$,
\item $q_{5,3}= x_1 x_2-1+x_4$.
\end{enumerate}
\end{multicols}
Moreover, since
\begin{itemize}
\item $S(q_{1,4},q_{2,5})=x_4\cdot q_{1,4}-x_2\cdot q_{2,5}=q_{3,1}-q_{5,3}\rightarrow_{B_1} 0$,
\item $S(q_{1,4},q_{3,1})=x_4x_5\cdot q_{1,4}-x_2x_3\cdot q_{3,1}=(x_5-1)\cdot q_{3,1}-(x_2-1)\cdot q_{1,4}\rightarrow_{B_1} 0 $,
\item $S(q_{1,4},q_{4,2})=x_5x_1\cdot q_{1,4}-x_2x_3\cdot q_{4,2}=(x_5-1)\cdot q_{4,2}-(x_3-1)\cdot q_{1,4}\rightarrow_{B_1} 0$,
\item $S(q_{1,4},q_{5,3})=x_1\cdot q_{1,4}-x_3\cdot q_{5,3}=q_{4,2}-q_{2,5}\rightarrow_{B_1} 0$,
\item $S(q_{2,5},q_{3,1})=x_5\cdot q_{2,5}-x_3\cdot q_{3,1}=q_{4,2}-q_{1,4}\rightarrow_{B_1} 0$,
\item $S(q_{2,5},q_{4,2})=x_5x_1\cdot q_{2,5}-x_3x_4\cdot q_{4,2}=(x_1-1)\cdot q_{4,2}-(x_3-1)\cdot q_{2,5}\rightarrow_{B_1} 0$,
\item $S(q_{2,5},q_{5,3})=x_1x_2\cdot q_{2,5}-x_3x_4\cdot q_{5,3}=(x_1-1)\cdot q_{5,3}-(x_4-1)\cdot q_{2,5}\rightarrow_{B_1} 0$,
\item $S(q_{3,1},q_{4,2})=x_1\cdot q_{3,1}-x_4\cdot q_{4,2}=q_{5,3}-q_{2,5}\rightarrow_{B_1} 0$,
\item $S(q_{3,1},q_{5,3})=x_1x_2\cdot q_{3,1}-x_4x_5\cdot q_{5,3}=(x_2-1)\cdot q_{5,3}-(x_4-1)\cdot q_{3,1}\rightarrow_{B_1} 0$,
\item $S(q_{4,2},q_{5,3})=x_2\cdot q_{4,2}-x_5\cdot q_{5,3}=q_{1,4}-q_{3,1}\rightarrow_{B_1} 0$,
\end{itemize}
then $B_1$ is a Gr\"obner basis for $I_4(C_5)$.
\end{Example}

The even case is slightly different.
\begin{Theorem}\label{grobnereven}
Let $n=2m\geq 6$ and
\[
B_0=\{q_{i,i+m}\, | \, i=0,\ldots, m-1\}\cup \{q_{i+m,i+1}\, | \, i=0,\ldots, m-2\}.
\]
Then $B_0$ is a Gr\"obner basis for $I_{n-1}(C_n,X)$ with respect to the graded lexicographic order with $x_{m-1}>\cdots>x_{2m}>x_1 \cdots>x_{m-2}$.
\end{Theorem}
\begin{proof}
We have that $q_{i,i+m},q_{i,i+m-1}\in B_0$ for all $i=1,\cdots,m-1$ and $q_{i,i-m},q_{i,i-m+1}\in B_0$ for all $i=m,\cdots,2m-2$.
Then by Claim~\ref{rel2}, $\{q_{i,j}\}_{j=1}^n\subsetneq \langle B_0\rangle$ for all $i=1,\cdots,2m-2$.
Moreover, since $q_{2m,1},q_{2m,2}\in \langle B_0\rangle$ and $q_{2m-1,1},q_{2m-1,2}\in \langle B_0\rangle$, 
then by Claim~\ref{rel2}, $\{q_{i,j}\}_{j=1}^n\subsetneq \langle B_0\rangle$ for $i=2m-1,2m$ and therefore $I_{n-1}(C_n, X)=\langle B_0\rangle$.
Following the proof of Theorem~\ref{Grobnerodd}, we have that 
${\rm lt}(q_{i,i+m})={\rm lt}(p_{i+1,m})=x_{i+1}\cdots x_{i+m-1}$ for all $i=0,\ldots, m-1$ and
${\rm lt}(q_{i+m,i+1})={\rm lt}(p_{i+m+1,m+1})=x_{i+m+1}\cdots x_{i}$ for all $i=0,\ldots, m-2$.
Thus ${\rm lt}(q_{i,i+m}) \nmid \,{\rm lt}(q_{i',i'+m})$ for all $0\leq i<i'\leq m-1$, 
${\rm lt}(q_{i+m,i+1}) \nmid \,{\rm lt}(q_{i'+m,i'+1})$ for all $0\leq i<i'\leq m-2$, and
${\rm lt}(q_{i,i+m}) \nmid \,{\rm lt}(q_{i'+m,i'+1})$ for all $0\leq i\leq m-1$ and $0\leq i'\leq m-2$.
Moreover, since ${\rm deg}({\rm red}(q_{i,i+m}))=m-1$, ${\rm deg}({\rm red}(q_{i+m,i+1}))=m-2$ and
${\rm lt}({\rm red}(q_{i,i+m}))={\rm lt}(p_{i+m+1,m})=x_{i+m+1}\cdots x_{i-1}$, 
when $B_0$ is a Gr\"obner basis for $I_{n-1}(C_n,X)$, then it is reduced.

In a similar way to the proof of Theorem~\ref{Grobnerodd}, we need to show that all the 
$S$-polynomials of the elements on $B_0$ can be reduced to $0$ by elements on $B_0$.
At difference to the proof of Theorem~\ref{Grobnerodd}, in this case
we have three types of $S$-polynomials of elements of $B_0$.
First
\begin{eqnarray*}
S(q_{i,i+m},q_{i+s,i+m+s}) &=& (x_{i+m}\cdots x_{i+m+s-1})\cdot q_{i,i+m}-(x_{i+1}\cdots x_{i+s})\cdot q_{i+s,i+m+s}\\
&=& {\rm lt}(p_{i+m,s+1})\cdot q_{i,i+m}-{\rm lt}(p_{i+1,s+1})\cdot q_{i+s,i+m+s}
\end{eqnarray*}
for all $0\leq i\leq m-2$ and $1\leq s\leq m-i-1$.
Applying Proposition~\ref{rel4} $(ii)$ with $j,t=m$ and $s=m+s$ we get that
\[
p_{i+m,s+1}\cdot q_{i,i+m}- p_{i+1,s+1}\cdot q_{i+s,i+m+s}=
p_{i+m+1,s}\cdot q_{i,i+m-1}-p_{i+1,s}\cdot q_{i+s+1,i+m+s}.
\]
Thus
\begin{eqnarray} \nonumber \label{eq01}
S(q_{i,i+m},q_{i+s,i+m+s})&=& -{\rm red}(p_{i+m,s+1})\cdot q_{i,i+m}+ {\rm red}(p_{i+1,s+1})\cdot q_{i+s,i+m+s}\\ 
&& +p_{i+m+1,s}\cdot q_{i,i+m-1}-p_{i+1,s}\cdot q_{i+s+1,i+m+s}.
\end{eqnarray}
Since $q_{i,j}=q_{j,i}$, then it is not difficult to use Equation~\ref{eq01} 
to prove that $S(q_{i,i+m},q_{i+s,i+m+s})$ is reduced to $0$ by $B_0$ whenever $i\geq 1$ and $i+s\leq m-2$.
Now, it only remains to analyze two special cases, when $i=0$ and $i+s=m-1$.
Taking $i=0$ in Equation~\ref{eq01} and the fact that
\[
q_{0,m-1}=q_{m-1,2m}\overset{\ref{rel2}}{=}p_{2m-1,2}\cdot q_{m-1,2m-1}-p_{2m,1}q_{m-1,2m-2}=x_{2m-1}\cdot q_{m-1,2m-1}-q_{m-1,2m-2}
\]
we have that
\begin{eqnarray*} 
S(q_{0,m},q_{s,m+s})
&=& -{\rm red}(p_{m,s+1})\cdot q_{0,m}+ {\rm red}(p_{1,s+1})\cdot q_{s,m+s} +x_{2m-1}\cdot p_{m+1,s}\cdot q_{m-1,2m-1}\\
&-&
\begin{cases}
p_{m+1,s}\cdot q_{m-1,2m-2}+ p_{1,s}\cdot q_{s+1,m+s} &\text{if } s\leq m-3,\\
(p_{m+1,m-2}+p_{1,m-2})\cdot q_{m-1,2m-2} &\text{if } s= m-2.
\end{cases}
\end{eqnarray*}

\noindent Therefore, $S(q_{0,m},q_{s,m+s}) \rightarrow_{B_0} 0$ for all $1\leq s\leq m-2$.
On the other hand, taking $i+s=m-1$ in Equation~\ref{eq01} and using that
$q_{m,2m-1}=q_{2m-1,m}\overset{\ref{rel2}}{=}x_{2m} \cdot q_{0,m}-q_{1,m}$, we have that
\begin{eqnarray*} 
S(q_{i,i+m},q_{m-1,2m-1})
&=& -{\rm red}(p_{i+m,s+1})\cdot q_{i,i+m}+ {\rm red}(p_{i+1,s+1})\cdot q_{m-1,2m-1}\\
&& +p_{i+m+1,s}\cdot q_{i,i+m-1}-p_{i+1,s}\cdot q_{m,2m-1}\\
&=& -{\rm red}(p_{i+m,s+1})\cdot q_{i,i+m}+ {\rm red}(p_{i+1,s+1})\cdot q_{m-1,2m-1}\\
&+&
\begin{cases} 
(p_{m+2,s}+p_{2,s})\cdot q_{1,m}- x_{2m}\cdot p_{2,s} \cdot q_{0,m} &\text{if } i= 1,\\
p_{i+m+1,s}\cdot q_{i,i+m-1}- x_{2m}\cdot p_{i+1,s} \cdot q_{0,m}+p_{i+1,s}\cdot q_{1,m} &\text{if } i\geq 2.
\end{cases}
\end{eqnarray*}
\noindent Therefore, $S(q_{i,i+m},q_{m-1,2m-1})$ can be reduced to $0$ by $B_0$.
Finally, if $i=0$ and $s=m-1$, then
\[
S(q_{0,m},q_{m-1,2m-1})= {\rm red}(p_{0,m})\cdot q_{m-1,2m-1}-{\rm red}(p_{m-1,2m-1})\cdot q_{0,m}.
\]

Now, we continue with the $S$-polynomials of a second type.
\begin{eqnarray*}
S(q_{i+m,i+1},q_{i+m+s,i+s+1}) &=& (x_{i+1}\cdots x_{i+s})\cdot q_{i+m,i+1}-(x_{i+m+1}\cdots x_{i+m+s})\cdot q_{i+m+s,i+s+1}\\
&=& {\rm lt}(p_{i+1,s+1})\cdot q_{i+m,i+1}-{\rm lt}(p_{i+m+1,s+1})\cdot q_{i+m+s,i+s+1}
\end{eqnarray*}
for all $0\leq i\leq m-3$ and $1\leq s\leq m-i-2$.
By Proposition~\ref{rel4} $(i)$ with $i=i+s+1$, $s=m-s-1$, and $t,j=m-1$ we get that
\[
p_{i+1,s+1}\cdot q_{i+m,i+1}-p_{i+m+1,s+1}\cdot q_{i+m+s,i+s+1}=
p_{i+2,s}\cdot q_{i+m,i}-p_{i+m+1,s}\cdot q_{i+m+s+1,i+s+1}
\]
Thus
\begin{eqnarray*} 
S(q_{i+m,i+1},q_{i+m+s,i+s+1})&=& -{\rm red}(p_{i+1,s+1})\cdot q_{i+m,i+1}+ {\rm red}(p_{i+m+1,s+1})\cdot q_{i+m+s,i+s+1}\\
&& +p_{i+2,s}\cdot q_{i+m,i}-p_{i+m+1,s}\cdot q_{i+m+s+1,i+s+1}
\end{eqnarray*}
\noindent for all $0\leq i\leq m-3$ and $1\leq s\leq m-i-2$.
Therefore, $S(q_{i+m,i+1},q_{i+m+s,i+s+1})\rightarrow_{B_0} 0$.
We finish with the $S$-polynomials of a third type.
\begin{eqnarray*}
S(q_{i,i+m},q_{i'+m,i'+1})&=&
\begin{cases}
(x_{i'+m+1}\cdots x_{i})\cdot q_{i,i+m}-(x_{i'+1}\cdots x_{i+m-1})\cdot q_{i'+m,i'+1} & \text{if } i\leq i'-1,\\
(x_{i'+m+1}\cdots x_{i'})\cdot q_{i,i+m}-(x_{i+1}\cdots x_{i+m-1})\cdot q_{i'+m,i'+1} & \text{if } i=i',\\
(x_{i'+m+1}\cdots x_{i'})\cdot q_{i,i+m}-(x_{i+1}\cdots x_{i+m-1})\cdot q_{i'+m,i'+1} & \text{if } i=i'+1,\\
(x_{i+m}\cdots x_{i'})\cdot q_{i,i+m}-(x_{i+1} \cdots x_{i'+m})\cdot q_{i'+m,i'+1}  & \text{if } i\geq i'+2,
\end{cases}
\\
&=&
\begin{cases}
{\rm lt}(p_{i'+m+1,m+i-i'+1})\cdot q_{i,i+m}-{\rm lt}(p_{i'+1,m+i-i'})\cdot q_{i'+m,i'+1} & \text{if } i\leq i',\\
{\rm lt}(p_{i+m,m+i'-i+2})\cdot q_{i,i+m}-{\rm lt}(p_{i+1,m+i'-i+1})\cdot q_{i'+m,i'+1}  & \text{if } i\geq i'+1,
\end{cases}
\end{eqnarray*}
for all $0\leq i\leq m-1$ and $0\leq i'\leq m-2$.
As in the previous cases, by Proposition~\ref{rel4} we get that
{\small
\begin{eqnarray*}
p_{i'+m+1,m+i-i'+1}\cdot q_{i,i+m}-p_{i'+1,m+i-i'}\cdot q_{i'+m,i'+1}=p_{i'+m+1,m+i-i'}\cdot q_{i+1,i+m}-p_{i'+2,m+i-i'-1}\cdot q_{i'+m,i'},\\
p_{i+m,m+i'-i+2}\cdot q_{i,i+m}- p_{i+1,m+i'-i+1}\cdot q_{i'+m,i'+1}=p_{i+m+1,m+i'-i+1}\cdot q_{i,i+m-1}-p_{i+1,m+i'-i}\cdot q_{i'+m+1,i'+1}.
\end{eqnarray*}
}
Finally, in a similar way as in the previous types of $S$-polynomials, 
it is not difficult to see that using these identities that $S(q_{i,i+m},q_{i'+m,i'+1})$ can be reduced  to zero by $B_0$ for all $0\leq i \leq m-1$ and $0\leq i' \leq m-2$. 
For instance, if $i=i',i'+1$, then
\[
S(q_{i,i+m},q_{i'+m,i'+1})= {\rm red}(p_{i,i+m+1})\cdot q_{i'+m,i'+1}-{\rm red}(p_{i'+m,i'+2})\cdot q_{i,i+m}.\vspace{-5mm}
\]
\end{proof}

\begin{Remark}
Note that Theorems~\ref{Grobnerodd} and~\ref{grobnereven} are independent of the base ring.
\end{Remark}
We present the case of the cycle with four vertices.
\begin{Example}
Let $n=4=2m$ and $B_0$ as in Theorem~\ref{grobnereven}.
It is not difficult to compute that $q_{1,3}= x_2 +x_4$, $q_{2,4}= x_1+x_3$, $q_{1,2}= x_3 x_4$, and
\begin{eqnarray*}
S(q_{1,3},q_{2,4})= x_4\cdot q_{2,4} -x_3\cdot q_{1,3} &\text{ with }& {\rm lt}(x_4\cdot q_{2,4})=x_1x_4\neq x_2x_3={\rm lt}(x_3\cdot q_{1,3}),\\
S(q_{1,3},q_{1,2})= x_4\cdot q_{1,2}, &\text{ and }& S(q_{2,4},q_{1,2})= x_3\cdot q_{1,2}.
\end{eqnarray*}

Therefore $B_0$ is a Gr\"obner basis for $I_3(C_4)$.
\end{Example}

Now, we present the cycle with six vertices.

\begin{Example}
Let $n=2m=2(3)=6$, then the polynomials $q_{6,3}=\underline{x_1x_2}+x_4x_5-2$, $q_{1,4}=\underline{x_2x_3 }+x_5x_6-2$, $q_{2,5}=\underline{x_3x_4}+x_1x_6-2$,
$q_{3,1}=\underline{x_4x_5x_6}-x_4-x_6+x_2$, $q_{4,2}=\underline{x_1x_5x_6}-x_1-x_5+x_3$
(the leading terms are underlined)
form a Gr\"obner basis for $I_{5}(C_6,X)$ with respect to the graded lexicographic order with $x_2>x_3>x_4>x_5>x_6>x_1$.
For instance, taking $i=0$ and $s=m-1=2$ it is not difficult to see that
\[
S(q_{6,3},q_{2,5})=x_3x_4\cdot q_{6,3}-x_1x_2\cdot q_{2,5}=(x_4x_5-2)\cdot q_{2,5}-(x_5x_6-2)\cdot q_{3,6}.
\]
Also, taking $i=1$ and $s=1$ we get that
\[
S(q_{1,4},q_{2,5})=x_4\cdot q_{1,4}-x_2\cdot q_{2,5}=x_4x_5x_6-2x_4-x_2x_1x_6+2x_2=q_{3,1}-q_{3,5}=2q_{3,1}-x_6\cdot q_{6,3}.
\]
\end{Example}

In order to finish this article we present a simple application of Theorem~\ref{mingen}, 
see Remark 6.8 in \cite{lorenzini08} for a similar result for graphs.
\begin{Corollary}
Let $D$ be a digraph and $v\in V(D)$ such that $D\setminus v\simeq C_n$, 
then $K(D)$ has at most two invariant factors different to one. 
\end{Corollary}

It is not easy to determine when the critical group of $G$ is cyclic because we need to evaluate 
a set of polynomials and after to compute its greatest common divisor.

\noindent {\bf Acknowledgments}

The authors would like to thank the anonymous referees for their helpful comments.


\end{document}